\documentclass[reqno, 11pt]{amsart}
\usepackage{amscd,amsfonts,amssymb}
\usepackage{graphicx}
\usepackage{color}
\usepackage{bm}
\usepackage[open]{bookmark} %generate outline. generate links to equations and references
\usepackage[nocompress]{cite} %auto ordering number in \cite{}
\usepackage{stmaryrd} %double square bracket for jump of functions

\usepackage[margin=1.2in]{geometry}  

\usepackage[nocompress]{cite}

\usepackage[colorinlistoftodos,prependcaption,textsize=tiny]{todonotes}

\providecommand{\abs}[1]{\left| #1 \right|}   %{\lvert#1\rvert} 
\providecommand{\norm}[1]{\left\| #1 \right\|}  %{\lVert#1\rVert}
\providecommand{\jump}[1]{\llbracket #1 \rrbracket}
\providecommand{\dualpair}[1]{ \langle #1 \rangle }

\numberwithin{equation}{section}

\newtheorem{Theorem}{Theorem}[section]
\newtheorem{Lemma}{Lemma}[section]

\newtheorem{Proposition}{Proposition}[section]

\theoremstyle{definition}
\newtheorem{Definition}{Definition}[section]
\theoremstyle{remark}
\newtheorem{Remark}{Remark}[section]

\DeclareMathOperator*{\esssup}{ess\,sup}

\author{Tian Jing}
\address{Department of Mathematics, University of Pittsburgh,
                           Pittsburgh, PA 15260, USA.}
\email{tij11@pitt.edu}

\title[Varifold Solutions of the two-phase 3-D MHD equations]
{Varifold solutions of the two-phase three-dimensional magnetohydrodynamic equations}


\numberwithin{equation}{section}
\numberwithin{figure}{section}
\numberwithin{table}{section}

\begin{document}

\begin{abstract}
In this paper we study the three-dimensional two-phase magnetohydrodynamic interface problem in a bounded domain. The two incompressible fluids are both Newtonian and the surface tension is considered.  We shall use the Galerkin method to construct the approximate solutions in a bounded domain. Due to the magnetic field  in the magnetohydrodynamic equations, we cannot use the method of monotone operators to solve the approximate equations. Instead, we will construct an iterating operator and solve the equations by finding the fixed-point of the operator. To deal with the free interface,  
we  shall  prove the compactness of the iterating operator and then use 
the Schauder fixed-point theorem. 
The existence of the varifold solution is established by the weak convergence method. 
\end{abstract}

\date{\today}

\maketitle

\section{Introduction and Main Results}\label{intro}

In this paper, we study the two-phase magnetohydrodynamic (MHD) problem of two immiscible Newtonian fluids which are incompressible, viscous  and conducting, 
%We study this problem 
in a three-dimensional bounded, 
simply connected smooth domain $\Omega\subseteq \mathbb{R}^3$. 
The domains of the two fluids are denoted by open sets $\Omega^+ (t)$ and $\Omega^- (t)$. 
The interface between them is defined as
 $\Gamma(t):=\partial\Omega^+ (t)\setminus \partial \Omega$.
The sets $\Omega^+ (t) $ ,  $\Omega^- (t)$ and $\Gamma(t)$ give a partition of $\Omega$.
We assume that the density equals to 1 everywhere 
and consider the following equations:
\begin{align}
\partial_{t}u+u\cdot\nabla u-(\nabla\times B)\times B-\nu^{\pm}\triangle u+\nabla p=0 & \ \ \ \text{in }\Omega^{\pm}(t)\label{eq:s,u}, \\
\partial_{t}B-\nabla\times(u\times B)+\nabla\times(\sigma\nabla\times B)=0 &  \ \ \ \text{in }\Omega\label{eq:s,B},  \\
{\rm div}u=0 &  \ \ \ \text{in }\Omega^{\pm}(t),  \\
{\rm div}B=0 &  \ \ \ \text{in }\Omega \label{divB 0},  \\
%(\nu^{+}-\nu^{-})n\cdot2Du=\kappa Hn &  \ \ \  \text{on }\Gamma(t)\label{eq:s,interface},  \\
 - \jump{ 2\nu(\chi) Du -pI }n = \kappa Hn &  \ \ \  \text{on }\Gamma(t)\label{eq:s,interface},  \\
V_{\Gamma}=n\cdot u &  \ \ \ \text{on }\Gamma(t)\label{eq:s,interface speed},  \\
u|_{\partial\Omega}=0, & \ B|_{\partial\Omega}=0,  \\
u|_{t=0}=u_{0}, & \ B|_{t=0}=B_{0},  \label{eq:s,initial}
\end{align}
where	
$ u\in\mathbb{R}^3$ is the velocity,
$ B\in\mathbb{R}^3$ the magnetic field,
$ \sigma>0$ the magnetic diffusion coefficient of both fluids, 
$\nu^+$, $ \nu^-\geq 0$ the viscosity coefficients 
of the two fluids,
$ \kappa\geq 0$ the surface tension coefficient; 
The quantities
$V_{\Gamma}$, $n$, $H$ are all defined pointwisely on the interface $\Gamma(t)$,
where 
$V_{\Gamma}$ denotes the velocity of the interface,
$n$ the normal vector,
$H$ the mean curvature;
The term
$Du:=(\nabla u+\nabla u^{T})/2$
is the strain rate tensor and $|Du|$ is the shear rate.
In order to study the positions of $\Omega^+ (t)$ and $\Omega^- (t)$, 
we consider the indicator function of $\Omega^+ (t)$, i.e. $\chi(t):=\chi_{\Omega^+ (t)}$.
Let $\nu$ be such that $\nu(1)=\nu^+$ and $\nu(0)=\nu^-$. 
Then we can use $\nu(\chi(t,x))$ for the viscosity.
 The notation $\jump{f}$ denotes the jump of  $f$  across $\Gamma(t)$.

We briefly review some related results.
When there is no magnetic field $B$, the problem becomes the two-phase Navier-Stokes equations. 
%It is well known that this interface problem is challenging.
The problem of varifold solutions was first studied by Plotnikov \cite{Plotnikov}.
In his paper, the case of two incompressible non-Newtonian fluids with surface tension 
has been considered in $\mathbb{R}^2$.
In the seminal work \cite{Abels1},  Abels   proved the existence of varifold solutions 
in more general cases, 
where the viscosity coefficients 
depend on the shear rate $\left|Du\right|$. 
%But in our problem the viscosity coefficients are both constants.
From \cite{Abels1}, there exists a weak solution when 
$\kappa=0$ and a measure-valued varifold solution when $\kappa>0$.
For the case of $\kappa>0$, the equations have been studied 
in $\mathbb{R}^2$ and $\mathbb{R}^3$. 
When the viscosity coefficients are constants, Yeressian \cite{Yere}
has proved the existence of varifold solutions  in $\mathbb{R}^3$.
%In his paper, the equations have been studied in $\mathbb{R}^3$.
When the strong solution exists, Fischer and Hensel  proved the weak-strong uniqueness in \cite{Fischer}
%The uniqueness is proved 
with the technique of relative entropy.
For interested readers we also refer to \cite{Salvi,Pruss,Nouri}.

Since the problem with $\kappa>0$ has been studied 
in $\mathbb{R}^2$ and $\mathbb{R}^3$
in \cite{Abels1} and  \cite{Yere}
for the Navier-Stokes equations,
in this paper we are interested in the problem in a bounded domain $\Omega$
for the magnetohydrodynamics, for which the  both viscosity coefficients $\nu^\pm$ are also taken to be  (different) constants.
In \cite{Abels1} and  \cite{Yere} the approximate equations are derived 
by mollifying the original equations.
In the case of a bounded domain, it will be complicated to mollify the equations near the boundary of the domain. 
Thus, we will use the Galerkin method to construct the approximate solutions in this paper.

We first give the definitions of varifold solutions and weak solutions
based on the definitions in \cite{Abels1}. The space $\mathbb{R}^d$ is replaced by  $\Omega$ in an appropriate way.
Some boundary conditions are also included.
\begin{Definition}[Varifold solution]
\label{def:vari solu} 
Let $u_{0}, B_{0}\in L^2(\Omega)$
such that ${\rm div}u_{0}={\rm div}B_{0}=0$ weakly. Let $Q_T := \Omega\times(0,T)$.
Let $\Omega_{0}^{+}\subseteq\Omega$ be a bounded domain such that
$\chi_{0}=\chi_{\Omega_{0}^{+}}$ is of finite perimeter. 
A quadruple $(u,B,\chi,V)$ with
\begin{align*}
&u\in L^2([0,T];H_0^1(\Omega))\cap L^{\infty}([0,T];L^2(\Omega)),\\
&B\in L^2([0,T];H_0^1(\Omega))\cap L^{\infty}([0,T];L^2(\Omega)),\\
&{\rm div}u={\rm div}B=0 , \\
&\chi\in L^{\infty}([0,T];BV(\Omega;\{0,1\})),\\
&V\in L^{\infty}([0,T];\mathcal{M}(\Omega\times\mathbb{S}^2)),
\end{align*}
is called a varifold solution to the two-phase flow
problem \eqref{eq:s,u}-\eqref{eq:s,initial} with the initial data $(u_0,B_0,\chi_0)$ if
\\
(1).
\begin{equation}   \label{def vari solu main}
\begin{aligned}
-&(u_{0},\varphi(0))_{\Omega}  -(u,\partial_{t}\varphi)_{Q_T}-(u\otimes u,\nabla\varphi)_{Q_{T}}+(B\otimes B,\nabla\varphi)_{Q_{T}}\\
 & +\left(2\nu(\chi)Du,D\varphi\right)_{Q_{T}}+\kappa\int_{0}^{T}\left\langle \delta V(t),\varphi(t)\right\rangle {d}t=0 
\end{aligned}
\end{equation}
is satisfied for all $\varphi\in C_{c}^{\infty}([0,T)\times\Omega)$ with ${\rm div}\varphi=0$;
\\
(2).
\begin{equation}
-(B_0,\varphi(0))_{\Omega}-(B,\partial_t \varphi)_{Q_T}
-(u\otimes B,\nabla\varphi)_{Q_T}+(B\otimes u,\nabla\varphi)_{Q_T}
+\sigma(\nabla B,\nabla\varphi)_{Q_T}
=0
\end{equation}
is satisfied for all $\varphi\in C_{c}^{\infty}([0,T)\times\Omega)$ with ${\rm div}\varphi=0$;
\\
(3). For almost every $t\in[0,T]$,
\begin{equation}
\begin{aligned}
\int_{\Omega\times\mathbb{S}^2}s\cdot\psi(x){d}V(t)(x,s) & =-\int_{\Omega}\psi{d}\nabla\chi(t)
\end{aligned}
\end{equation}
is satisfied for all  $\psi\in C_{0}(\Omega)$;
\\
(4).
The indicator function
$\chi$ is the unique
renormalized solution of 
\begin{equation}
\begin{aligned}
& \partial_{t}\chi+u\cdot\nabla\chi=0  \ \ \ \text{in }  \  (0,T)\times\Omega,\\
& \chi|_{t=0}=\chi_{0}   \ \ \  \text{in }   \  \Omega;
\end{aligned}
\end{equation}
\\
(5).
The generalized energy inequality 
\begin{equation}
\begin{aligned}
\frac{1}{2}\left\| u(t)\right\| _{L^2}^{2}+\frac{1}{2}\left\| B(t)\right\| _{L^2}^{2}+\kappa\left\| V(t)\right\| _{\mathcal{M}(\Omega\times\mathbb{S}^2)}+2\int_{0}^{t}\int_{\Omega}\nu(\chi)|Du|^2{d}x{d}s\\
+\sigma\int_0^t \left\| \nabla B(s)\right\| _{L^2}^{2}{d}s 
\leq\frac{1}{2}\left\| u_{0}\right\| _{L^2}^{2}+\frac{1}{2}\left\| B_{0}\right\| _{L^2}^{2}+\kappa\left\| \nabla\chi_{0}\right\| _{\mathcal{M}(\Omega)}
\end{aligned}
\end{equation}
holds for almost every $t\in[0,T]$;
\end{Definition}

\begin{Remark}
The notation $(\cdot ,\cdot )_{\Omega}$ and $(\cdot ,\cdot )_{Q_T}$ stands for 
the inner product in $L^2(\Omega)$ and $L^2(Q_T)$.
For details about the renormalized solutions, see Proposition 2.2 in \cite{Abels1}.
The term $\delta V $ in \eqref{def vari solu main} is the first variation of the measure $V$.
The definitions of $\delta V $ and $ \left\langle \delta V(t),\cdot\right\rangle $
are in Section \ref{prelim varifold}.
The initial energy is:
\begin{equation}
E_{0}:=
\frac{1}{2}\left\| u_{0}\right\| _{L^2}^{2}+\frac{1}{2}\left\| B_{0}\right\| _{L^2}^{2}+\kappa\left\| \nabla\chi_{0}\right\| _{\mathcal{M}}.
\end{equation}
\end{Remark}

\begin{Definition}[Weak solution]
Let $(u,B,\chi,V)$ be a varifold solution of the two-phase flow
problem \eqref{eq:s,u}-\eqref{eq:s,initial} with  the initial data $(u_0,B_0,\chi_0)$ 
as in Definition \ref{def:vari solu}.
Then the triple $(u,B,\chi)$ is called a weak solution if
for almost every  $t\in [0,T]$, the equality
\begin{align*}
\left\langle \delta V(t),\varphi\right\rangle =-\left\langle H_{\chi(t)},\varphi\right\rangle :=\int_{\Omega}P_{\tau}:\nabla\varphi{d}|\nabla\chi(t)|
\end{align*}
holds for all $\varphi\in C_{c}^{\infty}(\Omega)$ with ${\rm div}\varphi=0$. Here $P_{\tau}:=I-n\otimes n$ and
$n:=\nabla\chi(t)/\left|\nabla\chi(t)\right|$.
\end{Definition}

\begin{Remark}
The term $\chi$ contains all the information to define the mean curvature functional $H_{\chi}$;
see Section \ref{prelim mean curv} for details.
The term $\nabla\chi(t)$ is a vector-valued Radon measure on $\Omega$
and $\left|\nabla\chi(t)\right|$ is the total variation measure of $\nabla\chi(t)$.
Thus, the normal vector $n$ can be defined using
the Radon-Nikodym derivative.
See Section \ref{prelim} for details.
The varifold solution is weaker than the weak solution, 
since the weak limits of some terms are represented by measures.
\end{Remark}

The main result of this paper is given as the following:
\begin{Theorem}\label{main theorem}
Let $\Omega\subseteq \mathbb{R}^3$ be a bounded, smooth and simply connected domain;
$u_0, B_0 \in L^2(\Omega)$ satisfy ${\rm div}u_0= {\rm div}B_0 =0$;
and $\chi_{0}:=\chi_{\Omega_{0}^{+}}$, 
where $\Omega_{0}^{+}\subseteq\Omega$ is a 
simply connected $C^2$-domain such that $\overline{\Omega_0^+} \subseteq \Omega$.
Then for any $T>0$, there exists a 
varifold solution to the two-phase flow
problem \eqref{eq:s,u}-\eqref{eq:s,initial} on $[0,T]$ with the initial data $(u_0,B_0,\chi_0)$.
\end{Theorem}

The proof of Theorem \ref{main theorem} will be in the spirit of \cite{Abels1}
with some new ideas to deal with the bounded domain $\Omega$ and the magnetic field $B$.
We shall use the Galerkin method to construct the approximate solutions in a bounded domain.
Due to the extra term $B$ in the equations, we cannot use the method 
of monotone operators
in \cite{Abels1, Zeidler2A, Zeidler2B} to solve the approximate equations.
Instead, we will construct iterating operators and 
solve the equations by finding the fixed-points of the operators.
In fact, if our velocity $u$ is from certain function spaces,  
then the quantity $B$ and $\chi$ are uniquely decided by $u$.
Thus, there exist solution operators that map each $u$ 
to $B(u)$ and $\chi(u)$. These operators have some good properties of continuity and boundedness, which will contribute to showing the compactness of the iterating operator;
see \cite{Abels1} and \cite{hw} for more details. 
Due to the free interface $\Gamma(t)$, it is hard to prove 
the Lipschitz continuity of
the iterating operators.
Thus, we cannot use the classical contraction mapping theorem 
to prove the existence of the fixed-points.
In order to overcome this difficulty, we firstly prove the 
compactness of the iterating operators and then
use the Schauder fixed-point theorem.

The rest of this paper is organized as follows.
We firstly list some useful background knowledge in Section \ref{prelim}.
In Section \ref{Galerkin}, we will study the Galerkin approximate equations on $[0,T]$
and prove that the approximate solutions exists globally on $[0,T]$. 
Then we give a uniform energy estimate for all the approximate solutions.
Finally, we will study the limits of the approximate solutions 
in Section \ref{pass limit}.

\bigskip

\section{Preliminary}   \label{prelim}

\subsection{Function spaces}
We recall some definitions of function spaces.  % see \cite{Novotny} for details. 
Given a bounded domain  $\Omega\subseteq \mathbb{R}^d$.
The space $C^k(\Omega)$ denotes the functions with continuous partial derivatives until order $k$.
The subspace $C_b^k(\Omega)\subseteq C^k(\Omega)$ consists of bounded functions with bounded derivatives up to order $k$.
The space $C^k(\overline{\Omega})$ is the subspace of $C^k(\Omega)$, 
such that for each 
$f\in C^k(\overline{\Omega})$, we can find $F\in C^k(\mathbb{R}^d)$
with $f=F$ on $\overline{\Omega}$.
The space $C_{c,\sigma}^{\infty}(\Omega)$ consists of functions in $C_{c}^{\infty}(\Omega)$  which are divergence-free.
The following result will be useful.
\begin{Proposition}[\cite{Zeidler1}, Appendix (24d)] 
\label{compact embedding Ck space}
For a compact set $K\subseteq\mathbb{R}^d$, 
for any $k\in\mathbb{Z}$ and $k\geq 0$,
we have the compact embedding
$C^{k+1}(K)\hookrightarrow \hookrightarrow C^{k}(K) $.
\end{Proposition}
For a Banach space $X$ and $1\leq p\leq \infty $, the Bochner space
$L^p([0,T];X)$ is the space of functions $u(t)$ from $[0,T]$ to $X$
such that
$$
\left\| u \right\|_{L^p ([0,T];X)}
:=
\left( \int_0^T\left\| u(t)   \right\|_X^p  {d}t \right)^{\frac{1}{p}} < \infty  
\text{ \  \  for \ }  1\leq p <\infty \text{, or}
$$
$$
\left\|  u  \right\|_{L^{\infty} ([0,T];X)}
:=
\esssup_{t\in[0,T]} \left\| u(t)   \right\|_X < \infty
\text{ \  \  for \ } p =\infty.
$$
The space
$C([0,T];X)$ consists of functions $u(t)$ from $[0,T]$ to $X$
such that for any $t_0\in[0,T]$
$$
\lim_{t\to t_0} \left\| u(t)-u(t_0)   \right\|_X =0. 
$$
For more details on function spaces, we refer to \cite{Novotny}.
For Sobolev spaces and the embedding theorems, we refer to \cite{Evans1, Leoni}.

For a locally compact separable metric space $X$, we denote the space of finite Radon measures on $X$ by $\mathcal{M}(X)$. We recall that $\mathcal{M}(X)= (C_0(X))^*$.

We denote by $L_w^{\infty}([0,T],\mathcal{M}(X))$ the space of 
functions $f:[0,T]\to \mathcal{M}(X)$ such that: 
for all $\varphi \in L^1([0,T],C_0(X))$ the duality
$(f,\varphi)(t)$ is measurable on $[0,T]$;
$\left\| f \right\|_{\mathcal{M}(X)}(t)$ is measurable on $[0,T]$;
and $\esssup_{[0,T]} \left\| f \right\|_{\mathcal{M}(X)}(t) < \infty$.
More details can be found in \cite{Abels1,Yere}.

\subsection{Sets of finite perimeter}

In order to study the free interface, we recall some topics about the sets with less regular boundaries. 
Most of the topics here can be found in the geometric measure theory. 
We refer to \cite{Amb.BV, Maggi} for interested readers.
\begin{Definition}[\cite{Amb.BV}, Definition 3.1]
For an open set $\Omega\subseteq\mathbb{R}^d$. A function $u\in L^{1}(\Omega)$ is of bounded variation in $\Omega$
if there exist
finite Radon measures $\lambda_1,\cdots \lambda_d \in\mathcal{M}(\Omega)$ such that
\[
\int_{\Omega}u\frac{\partial\varphi}{\partial x_{i}}{d}x=-\int_{\Omega}\varphi{d}\lambda_i
\]
holds
for all $\varphi\in C_{c}^{\infty}(\Omega)$ and $i=1,\cdots,d$.
The space ${BV}(\Omega)$ consists of functions of bounded variation in $\Omega$.
\end{Definition}

The variations of functions are useful when we study the space ${BV}(\Omega)$.
\begin{Definition}[\cite{Amb.BV}, Definition 3.4]
Given a function $u\in L_{loc}^{1}(\Omega)^{m}$ with $\Omega\subseteq\mathbb{R}^d$.
Its variation in $\Omega$ is defined as
\[
\mathcal{V}(u,\Omega):=\sup\left\{\sum_{\alpha=1}^{m}\int_{\Omega}u^{\alpha} {\rm div} \varphi^{\alpha}{d}x:\varphi\in C_{c}^{1}(\Omega)^{md},\,\,\left\| \varphi\right\| _{L^{\infty}(\Omega)}\leq1\right\}.
\]
\end{Definition}
 
When studying the boundary of a set $E$, 
we usually consider its indicator function $\chi_E$. 
The perimeter of $E$ is defined using the variation of $\chi_E$.
See \cite{Amb.BV, Evans3} for details.
\begin{Definition}[\cite{Amb.BV}, Definition 3.35]
Let $\Omega\subseteq\mathbb{R}^{d}$ be an open set and
$E\subseteq \mathbb{R}^{d}$ a Lebesgue measureable set. 
The perimeter of $E$
in $\Omega$ is defined as:
\[
\mathcal{P}(E,\Omega):=\sup\left\{\int_{E} {\rm div} \varphi{d}x:\varphi\in C_{c}^{1}(\Omega)^{d},\,\,\left\| \varphi\right\| _{L^{\infty}(\Omega)}\leq1\right\}.
\]
The set $E$ is of finite perimeter in $\Omega$ if $\mathcal{P}(E,\Omega)<\infty$.
\end{Definition}

\begin{Remark}
By this definition, only $\partial E\cap \Omega$ will be counted
into the perimeter of $E$.
\end{Remark}

\subsection{Mean curvature functional} \label{prelim mean curv}
When calculating the weak form of \eqref{eq:s,u},
we will get a so-called mean curvature functional.
This functional is dependent on the interface $\Gamma(t)$, and thus will be denoted by 
$H_{\Gamma(t)}$ or $H_{\chi(t)}$.
When everything is smooth enough, we have the following formula:
\begin{align*}
\left\langle H_{\Gamma(t)},\varphi(t)\right\rangle :=\int_{\Gamma(t)}Hn\cdot\varphi{d}\mathcal{H}^{d-1}(x),
\end{align*}
where $\mathcal{H}^k$ denotes the $k$-dimensional Hausdorff measure.
On the interface $\Gamma$  
the tangential divergence of a function $\varphi\in C^{1}(\Gamma)^{d}$
is defined as
$$
{\rm div}^{\Gamma}\varphi:={\rm div}\varphi- n\otimes n:\nabla\varphi
= (I-n\otimes n):\nabla\varphi,
$$
where $n(x)$ denotes the normal vector. 
Note that 
$P_{\tau}:=I-n\otimes n$ is the orthogonal projection onto the tangent space,
which is defined pointwisely on $\Gamma$.
The mean curvature is defined as (see \cite{Amb.BV, Maggi, Simon})
$$
H:=-{\rm div}^{\Gamma}\varphi
$$
and the mean curvature vector is $\bm{H}:=Hn$.
According to the generalized divergence theorem (see \cite{Maggi}),
\begin{equation}\label{generalized divergence theorem}
\int_{\Gamma}{\rm div}^{\Gamma}\varphi{d}\mathcal{H}^{d-1}
=
\int_{\Gamma}Hn\cdot\varphi{d}\mathcal{H}^{d-1}
+
\int_{\partial\Gamma}\varphi\cdot n_{\partial\Gamma} {d}\mathcal{H}^{d-2}
\end{equation}
holds for all $\varphi\in C_c^1(\mathbb{R}^d)^d$.
In our problem, the set  $\overline{\Omega^+ (t)}$ will always stay in 
the interior of $\Omega$. Thus, $\Gamma(t)$ will be a closed surface 
and will not have any edge, i.e. the $(d-2)$-dimensional boundary.
As a result, the second term on the right hand side in \eqref{generalized divergence theorem} vanishes.
Now when the surface $\Gamma$ becomes less regular, 
as long as a measure theoretical normal vector exists,
we can use the generalized divergence theorem to 
replace $Hn\cdot\varphi$ with ${\rm div}^{\Gamma}\varphi$,
and study the generalized mean curvature functional.

In order to study the convergence of the mean curvature terms, 
we modify Lemma 2.4 from \cite{Abels1} and obtain the following lemma.
\begin{Lemma}[\cite{Abels1}, Lemma 2.4, revised]
Let $\Omega$ be a bounded, simply connected, smooth domain.
Let $\Omega_{0}^{+}$ be a bounded, simply connected $C^2$-domain such that 
$\overline{\Omega_{0}^{+}}\subseteq \Omega$.
Suppose that $u$, $v\in C([0,T];C_{b}^{2}(\Omega))$ with $ {\rm div} u= {\rm div} v=0$
and $u\to v$ in $C([0,T];C^{1}(\Omega))$.
Then
\begin{equation}
\int_{\Gamma_{u}(t)}f(x,n_{x}){d}\mathcal{H}^{d-1}(x)
\to
\int_{\Gamma_{v}(t)}f(x,n_{x}){d}\mathcal{H}^{d-1}(x)
\end{equation}
uniformly on $[0,T]$.
Here $\Gamma_{u}(t)$ and $\Gamma_{v}(t)$ 
are interfaces obtained from $u$ and $v$.
\end{Lemma}
The details on how the velocity determines the interface 
are stated in Section \ref{solu operators}.
The convergence of the flow mappings is still valid 
when the domain $\mathbb{R}^d$ in \cite{Abels1}
is replaced by a bounded domain $\Omega$.
Thus, using the same argument as in \cite{Abels1},
we can apply a local parameterization to 
$\Gamma_0=\partial\Omega_0^+$
to prove the lemma.

\subsection{Varifolds} \label{prelim varifold}
For the problem we study, even the measure theoretical normal vectors might not be guaranteed to exist all the time.
In this case, we have to use varifolds to describe the surfaces.
We refer to the definition in \cite{Abels1} 
and give an analogue one for the case of the bounded domain $\Omega$.
A measure $V$ is called a general $(d-1)$-varifold
if it is a finite Radon measure on $\Omega\times\mathbb{S}^{d-1}$,
i.e. $V\in\mathcal{M}(\Omega\times\mathbb{S}^{d-1})$. 
The varifold $V$ can be understood as assigning different weight to vectors in 
$\Omega$ and $\mathbb{S}^{d-1}$.
In another word, it tells the possibility of a point to be on the interface 
and the possibility of a vector to be the normal vector.
For interested readers we refer to \cite{Morgan, Simon, Malek}. 

%Define $|V|(E):=V(E\times \mathbb{S}^{d-1})$ for any $E\in \mathcal{B}(\Omega)$.
%Using the disintegration theorem,
%we can map each $x\in\Omega$ to a Radon measure 
%$V_x\in\mathcal{M}(\mathbb{S}^{d-1})$
%such that $|V_x|(\mathbb{S}^{d-1})=1$ for $|V|$-almost every $x\in\Omega$.
%Then we can decompose the integral as
%$$
%\langle V, \varphi\rangle
%:=
%\int_{\Omega \times \mathbb{S}^{d-1}} \varphi(x,s){d}V
%=
%\int_{\Omega}\int_{\mathbb{S}^{d-1}} \varphi(x,s){d}V_x{d}|V|.
%$$
%For the disintegration theorem and related knowledge, we refer to \cite{Amb.BV}.
%Theorem 2.28. 

The first variation of $V$ is defined as
$$
\langle\delta V, \varphi\rangle
:=
\int_{\Omega\times\mathbb{S}^{d-1}} (I-s\times s):\nabla\varphi{d}V(x,s)
$$
for any $\varphi\in C_0^1(\Omega)$; see \cite{Simon} chapter 8 or \cite{Allard}.
This allows us to replace the functional $H_{\chi_E (t)}$ with $-\delta V(t)$,
and study the mean curvature functional when the interface is less regular.

\subsection{Compact operators} 
We recall some theory about compact operators;
see \cite{Zeidler1} for details. 
The compact operators will be used to solve the Galerkin approximate equations.
\begin{Definition}[\cite{Zeidler1}, Definition 2.9]
\label{def compact operator} 
Given two Banach spaces $X$ and $Y$. An operator $T: D(T)\subset X\to Y$
is called a compact operator if
 it is continuous 
and maps bounded sets into precompact sets.
\end{Definition}
The following proposition is important when proving the compactness of an operator.
\begin{Proposition}[\cite{Zeidler1}, Appendix (24g)] %p772
\label{Zeidler1 Appendix 24g}
The set $M\subseteq C(\overline{\Omega})$ is precompact if and only if \\
(1)
$\sup_{f\in M}\sup_{x\in\overline{\Omega}}|f(x)|<\infty$. \\
(2)  For every $\varepsilon>0$, there exists $\delta>0$,
such that 
 $\sup_{f\in M}|f(x)-f(y)|<\varepsilon$
for every $x,y\in\overline{\Omega}$
and $|x-y|<\delta$. 
\end{Proposition}
We give a specific version of the Arzela-Ascoli theorem.
\begin{Theorem}[\cite{Zeidler1}, Appendix (24i)]\label{Arzela Ascoli}
Let $X$ be a Banach space.
The set $A\subseteq C([0,T];X)$ is precompact if and only if \\
(1) For all $t\in[0,T]$, the set $\{f(t):f\in A\}$ is precompact in $X$.\\
(2) For all $t\in[0,T]$ and $\varepsilon>0$, there exists $\delta>0$,
such that $\sup_{f\in A}\left\| f(t)-f(s)\right\| _{X}<\varepsilon$
for all $s\in[0,T]$ and $|t-s|<\delta$. 
\end{Theorem}
At last, we recall the Schauder fixed-point theorem of compact operators.
\begin{Theorem}[\cite{Zeidler1}, Theorem 2.A]\label{fixed point thm}
Let  $X$ be a Banach space. Suppose $A\subseteq X$ is nonempty,
bounded, closed and convex. Given a compact operator $T:A\to A$. 
There exists  a  fixed point of $T$ in $A$. 
\end{Theorem}

\bigskip

\section{The Galerkin Method}\label{Galerkin}

\subsection{Weak form and energy estimate}

We test (\ref{eq:s,u}) with $\varphi\in C_c^{\infty}([0,T)\times\Omega)$
such that
 ${\rm div}\varphi=0$. 
 For the first and second terms we simply integrate them by parts; 
 for the third term we recall the equality $(\nabla\times B)\times B= B\cdot\nabla B -\nabla (|B|^2) /2$ and then integrate by parts;  
  the fifth term will vanish;
the calculation of the fourth term will generate the mean curvature functional:
\begin{align*}
 & \int_{\Omega}\nu(\chi)\triangle u\varphi=\int_{\Omega}\nu(\chi){\rm div}(\nabla u+\nabla u^{T})\varphi\nonumber \\
= & \int_{\Omega}\nu(\chi)\sum_{j}\sum_{i}\partial_{i}(\partial_{i}u_{j}+\partial_{j}u_{i})\varphi_{j}\\
= & \int_{\Omega}\nu(\chi)\sum_{j}\sum_{i}\partial_{i}((\partial_{i}u_{j}+\partial_{j}u_{i})\varphi_{j})-\int_{\Omega}\nu(\chi)\sum_{j}\sum_{i}(\partial_{i}u_{j}+\partial_{j}u_{i})\partial_{i}\varphi_{j}\nonumber \\
= & 2 (\nu^{+}-\nu^{-})\int_{\Gamma}n\cdot (Du \varphi) - 2\int_{\Omega}\nu(\chi)Du:D\varphi\nonumber \\
= & \kappa\int_{\Gamma}Hn\cdot\varphi-2(\nu(\chi)Du,D\varphi)_{\Omega}. 
\end{align*}
Thus, we obtain the weak formula of \eqref{eq:s,u}:
\begin{equation}\label{eq:weak,u}
\begin{aligned}
- & (u_0,\varphi(0))_{\Omega} - (u,\partial_t \varphi)_{Q_T}  -(u\otimes u,\nabla\varphi)_{Q_{T}}+(B\otimes B,\nabla\varphi)_{Q_{T}}\\
 & +2(\nu(\chi)Du,D\varphi)_{Q_{T}}-\kappa\int_{0}^{T}\int_{\Gamma(t)}Hn\cdot\varphi{d}\mathcal{H}^{2}=0.
\end{aligned}
\end{equation}
Testing \eqref{eq:s,B} with  $\varphi\in C_c^{\infty}([0,T)\times\Omega)$ 
such that ${\rm div}\varphi=0$. Using the fact that 
$$\nabla\times(\nabla\times B)=\nabla({\rm div}B) -\triangle B, \;  \text{ and } \;
\nabla\times(u\times B)=-(u\cdot\nabla)B+(B\cdot\nabla)u,$$
we obtain
\begin{equation}\label{eq:weak,B}
-(B_0,\varphi(0))_{\Omega}-(B,\partial_t \varphi)_{Q_T}
-(u\otimes B,\nabla\varphi)_{Q_T}+(B\otimes u,\nabla\varphi)_{Q_T}
+\sigma(\nabla B,\nabla\varphi)_{Q_T}  = 0.
\end{equation}
Now we derive the energy estimate. Suppose all the functions are smooth enough.
Testing \eqref{eq:s,u} and \eqref{eq:s,B} on $\Omega$ with $\varphi=u$
and $\varphi=B$ respectively, we obtain
\begin{align*}
\frac{1}{2}\frac{{d}}{{d}t}\left\| u\right\| _{L^2}^{2}+(B\otimes B,\nabla u)_{\Omega} + 2(\nu(\chi)Du,Du)_{\Omega}-\kappa\int_{\Gamma(t)}Hn\cdot u{d}\mathcal{H}^2=0,
\end{align*}
\begin{align*}
\frac{1}{2}\frac{{d}}{{d}t}\left\| B\right\| _{L^2}^{2}-(B\otimes B,\nabla u)_{\Omega}+\sigma\left\| \nabla B\right\| _{L^2}^{2}=0.
\end{align*}
Adding these two equations, we have
\begin{align*}
\frac{1}{2}\frac{{d}}{{d}t}\left\| u\right\| _{L^2}^{2}+\frac{1}{2}\frac{{d}}{{d}t}\left\| B\right\| _{L^2}^{2}+2(\nu(\chi)Du,Du)_{\Omega}\\
+\sigma\left\| \nabla B\right\| _{L^2}^{2}-\kappa\int_{\Gamma(t)}Hn\cdot u{d}\mathcal{H}^2=0.
\end{align*}
From the derivation of (1.9) in \cite{Abels1}, we have
\[
\frac{{d}}{{d}t}\mathcal{H}^2(\Gamma(t))=-\int_{\Gamma(t)}HV_{\Gamma}{d}\mathcal{H}^2=-\int_{\Gamma(t)}Hn\cdot u{d}\mathcal{H}^2.
\]
Note that by Korn's inequality, there exists $c>0$ such that
\[
2(\nu(\chi)Du,Du)_{\Omega}\geq c\left\| \nabla u\right\| _{L^2}^{2}.
\]
Finally, we obtain the energy inequality:
\begin{equation}
\begin{aligned}
\frac{1}{2}\left\| u(t)\right\| _{L^2}^{2}+\frac{1}{2}\left\| B(t)\right\| _{L^2}^{2}+\kappa\mathcal{H}^2(\Gamma(t))+c\left\| \nabla u\right\| _{L^2([0,T]\times\Omega)}^{2}+\sigma\left\| \nabla B\right\| _{L^2([0,T]\times\Omega)}^{2}\\
\leq\frac{1}{2}\left\| u_{0}\right\| _{L^2}^{2}+\frac{1}{2}\left\| B_{0}\right\| _{L^2}^{2}+\kappa\mathcal{H}^2(\Gamma_{0}).
\end{aligned}
\end{equation}
This estimate drives us to look for a solution $(u,B,\Gamma)$
such that 
$$u\in L^2 ([0,T];H_0^1(\Omega)) \cap L^{\infty}([0,T];L^2(\Omega)),\quad 
 B \in L^2 ([0,T];H_0^1(\Omega)) \cap L^{\infty}([0,T];L^2(\Omega)),$$
and
$\mathcal{H}^2(\Gamma(t))$ is bounded on $[0,T]$.

\subsection{Approximate equations}
In order to use the Galerkin method, 
we pick the eigenfunctions of the Stokes operator to be a basis.
The existence of this basis is from the following theorem:
\begin{Theorem}[\cite{3dNS}, Theorem 2.24]
\label{thm:Galerkin basis}
%[Robinson, Rodrigo and Sadowski, p60,Theorem 2.24]
Let $\Omega\subseteq\mathbb{R}^3$ be a smooth bounded domain. 
Let $A$ be the Stokes operator, 
i.e. $Au:=-\mathbb{P}\triangle u$ 
where $\mathbb{P}$ is the Helmholtz projection. 
There exists a set of functions 
$\mathcal{N}=\{\eta_{1},\eta_{2},\cdots,\}$
such that \\
(1) the functions form an orthonormal basis of $\mathbb{H}(\Omega)$; \\
(2) the functions form an orthogonal basis of $\mathbb{V}(\Omega)$; \\
(3) the functions belong to $D(A)\cap C^{\infty}(\overline{\Omega})$ and they are eigenfunctions
of $A$ with positive, nondecreasing eigenvalues which goes to infinity.
\end{Theorem}

Here $\mathbb{H}(\Omega)$ denotes the closure 
of $\{ \varphi\in C_c^{\infty}(\Omega) : {\rm div}\varphi=0 \}$ 
under the $L^2$ norm
and $\left\|\cdot\right\|_{\mathbb{H}(\Omega)}:=\left\|\cdot\right\|_{L^2(\Omega)}$.
The space $\mathbb{V}(\Omega):=\mathbb{H}(\Omega)\cap H_0^1(\Omega)$
and 
$\left\|\cdot\right\|_{\mathbb{V}(\Omega)}:=\left\|\cdot\right\|_{H^1(\Omega)}$.
Note that 
%by Theorem \ref{thm:Galerkin basis},
all the eigenfunctions $\eta_{j}$ have trace 0.
For details of the Stokes operator, see \cite{Temam2, Temam1}.

Let $G_n:=\text{span}\{\eta_1,\cdots, \eta_n\}$.
For each $t\in[0,T]$, we consider the approximate equation:
\begin{equation}\label{eq: Galerkin u}
\begin{aligned}
(u_{n}(t),\eta)_{\Omega} & - (u_{0},\eta)_{\Omega} -\int_0^t (u_{n}\otimes u_{n},\nabla\eta)_{\Omega}{d}s
+
\int_0^t (B_{n}\otimes B_{n},\nabla\eta)_{\Omega}{d}s\\
+  & 2\int_0^t (\nu(\chi_n)Du_n,D\eta)_{\Omega}{d}s
+
\kappa\int_0^t \int_{\Omega}P_{\tau}:\nabla\eta{d}\left|\nabla\chi_{n}(s)\right|{d}s = 0,   
\end{aligned}
\end{equation}
for all $\eta\in G_n$. 
We define the functionals $M$ and $N$ on $G_n$ 
and rewrite the equation. Let
$$
\left\langle M(u),\eta\right\rangle   :=\int_{\Omega}u\cdot\eta,
$$
\begin{align*}
\left\langle  N(u,\chi,B),\eta\right\rangle & :=(u\otimes u,\nabla\eta)_{\Omega}-(B\otimes B,\nabla\eta)_{\Omega}\\
& - 2(\nu(\chi)Du,D\eta)_{\Omega}+\kappa\int_{\Omega}P_{\tau}:\nabla\eta{d}\left|\nabla\chi\right|.
\end{align*}
By integrating $N(u,\chi,B)$ from $0$ to $t$, we define
$$
\left\langle \int_{0}^{t}N{d}s,\eta\right\rangle   :=\int_{0}^{t}\left\langle N(s),\eta\right\rangle {d}s.
$$
Now we can rewrite the equation as:
\begin{equation}
\left\langle M(u_{n}(t)),\eta\right\rangle =\left\langle M(u_{0}),\eta\right\rangle +\int_{0}^{t}\left\langle N(u_{n},\chi_{n},B_{n}),\eta\right\rangle {d}s
\end{equation}
for all $\eta\in G_n$.
It remains to represent $\chi$ and $B$ 
with $u$ using the solution operators,
i.e.  $\chi(u)$ and $B(u)$ .

\subsection{Solution operators $\chi(u)$ and $B(u)$}
\label{solu operators}
Suppose $u\in C([0,T];C^{2}(\overline{\Omega}))$ 
and $\Omega_0^+$ is a simply connected $C^2$-domain with $\overline{\Omega_0^+} \subseteq \Omega$.
For each $x\in\overline{\Omega}$, we consider the ODE
\begin{equation}
\begin{aligned}
&\frac{{d}}{{d}t}X(t,x) =  u\left(t,X(t,x)\right), \\
& X(0,x)=  x.
\end{aligned}
\end{equation} 
By the Picard-Lindelöf theorem there exists a unique solution locally in time.
Since the solution will not blow up as stated in Remark \ref{ode not blow up}, 
we can always extend it to $[0,T]$.
When we start from different initial values on $\overline{\Omega}$,
the solutions will not intersect. 
Thus, we obtain a function 
$X(t,x):[0,T]\times \overline{\Omega} \to \overline{\Omega}$,
which is a bijection on $\overline{\Omega}$ for each fixed $t$.
We call $X(t,x)$ the flow mapping, and denote it by $X_t(x)$ in some cases.
We will also use $X_u(t,x)$ or $X_{u,t}(x)$ if needed to emphasize 
the velocity field that generates this flow mapping.
\begin{Remark}\label{ode not blow up}
When $x\in\partial\Omega$, 
the Picard iterating always generate constant functions equal to $x$.
Thus, we can obtain a unique solution $X(t,x)\equiv x$ on $[0,T]$.
When $x\in \Omega$, the local solution will not exceed $\Omega$, 
so it can still be extended to $[0,T]$.
In both cases, the proof of uniqueness can be done by the Gronwall's inequality.
\end{Remark}
Note that $u\in C([0,T];C_0^2(\overline{\Omega}))$.
Similarly to the proof of Theorem 2.10 in \cite{Teschl},  
we can prove that 
$X\in C([0,T];C^2(\Omega))$.
Letting $\chi(x,t):=\chi_0(X_t^{-1}(x))$, 
then we have obtained the indicator function $\chi$ using the velocity $u$.

We now estimate the variation of $\chi(x,t)$.
From \cite{Amb.TE} Exercise 3.2, the Jacobian $J(X_{t})\equiv 1$.
By changing of variable, we have
\begin{align}\label{estimate chi BV}
\int_{\Omega}\chi(x,t){\rm div}\varphi(x) {d}x
= \int_{\Omega}\chi(X_{t}(y),t){\rm div}\varphi(X_{t}(y)){d}y.
\end{align}
Similarly to the argument in \cite{Abels1}, 
we integrate by parts.
%Considering the gradient of $X_t$:
%\begin{align*}
%\nabla X_{t} & 
%=
%\begin{pmatrix}
%\partial_{1}X_{t,1} & \partial_{1}X_{t,2} & \partial_{1}X_{t,3}\\
%\partial_{2}X_{t,1} & \partial_{2}X_{t,2} & \partial_{2}X_{t,3}\\
%\partial_{3}X_{t,1} & \partial_{3}X_{t,2} & \partial_{3}X_{t,3}
%\end{pmatrix},
%\end{align*}
Let $A=(a_{ij})_{3\times 3}$ be the matrix inverse of $\nabla X_{t}$, i.e.
%and its inverse matrix %and the transpose
$
A(y):=
(\nabla_y X_{t}(y))^{-1}.
%=
%\begin{pmatrix}
%a_{11}(y) & a_{12}(y) & a_{13}(y)\\
%a_{21}(y) & a_{22}(y) & a_{23}(y)\\
%a_{31}(y) & a_{32}(y) & a_{33}(y)
%\end{pmatrix}.
$
Let
$
\widetilde{\varphi}(y)  =
A^T(y) \varphi(X_{t}(y)) 
%\begin{pmatrix}
%\sum_{j=1}^{3}a_{j1}\varphi_{j}(X_{t})\\
%\sum_{j=1}^{3}a_{j2}\varphi_{j}(X_{t})\\
%\sum_{j=1}^{3}a_{j3}\varphi_{j}(X_{t})
%\end{pmatrix}.
$
with $A^T$ being the transpose of $A$.
For the gradient of $\widetilde{\varphi}(y) $, i.e.
$\nabla_y \left(A^T(y)\varphi(X_{t}(y))\right)$,
%\begin{align*}
%\begin{pmatrix}
%\partial_{y_{1}} \widetilde{\varphi}_1(y)
%& \partial_{y_{2}} \widetilde{\varphi}_1(y)
% &\partial_{y_{3}} \widetilde{\varphi}_1(y)\\
%\partial_{y_{1}} \widetilde{\varphi}_2(y)
%& \partial_{y_{2}} \widetilde{\varphi}_2(y)
% &\partial_{y_{3}} \widetilde{\varphi}_2(y)\\
%\partial_{y_{1}} \widetilde{\varphi}_3(y)
%& \partial_{y_{2}} \widetilde{\varphi}_3(y)
% &\partial_{y_{3}} \widetilde{\varphi}_3(y)\\
%\end{pmatrix}.
%\end{align*}
we consider its trace:
\begin{align*}
& \text{Tr} \nabla_y \left(A^T\varphi(X_{t}(y))\right)\\ & =\sum_{i}\sum_{j}\partial_{y_{i}}a_{ji}\cdot\varphi_{j}(X_{t}(y))+\sum_{i}\sum_{j}a_{ji}\cdot\partial_{y_{i}}\varphi_{j}(X_{t}(y))\\
& = I_{1}+I_{2}.
\end{align*}
Simplifying $I_2$, we obtain
\begin{align*}
I_{2} & =\sum_{i}\sum_{j}\sum_{k}a_{ji}\partial_{k}\varphi_{j}(X_{t}(y))\cdot\partial_{y_{i}}X_{t,k}(y)\\
= & \sum_{j}\sum_{k}\partial_{k}\varphi_{j}(X_{t}(y))\cdot\left(\sum_{i}a_{ji}\partial_{y_{i}}X_{t,k}(y)\right)\\
= & \sum_{j}\sum_{k}\partial_{k}\varphi_{j}(X_{t}(y))
\cdot\left(\sum_{i}\left((\nabla X_t)^{-1}\right)_{ji}\left(\nabla X_{t}\right)_{ik}\right)\\
= & \sum_{j}\partial_{j}\varphi_{j}(X_{t}(y))={\rm div}\varphi(X_{t}(y)).
\end{align*}
Continuing with \eqref{estimate chi BV}, we have
\begin{align*}
\int_{\Omega}\chi & (X_{t}(y),t){\rm div}\varphi(X_{t}(y)){d}y
=\int_{\Omega}\chi  (X_{t}(y),t) I_2 {d}y
\\
= & \int_{\Omega}\chi(X_{t}(y),t)\text{Tr}\nabla_y (A^T \varphi(X_{t}(y)))-\int_{\Omega}\chi(X_{t}(y),t)I_1
\\
= & \int_{\Omega}\chi(X_{t}(y),t)\text{Tr}\nabla_y ( A^T \varphi(X_{t}(y)))-\int_{\Omega}\chi(X_{t}(y),t)\sum_{i}\sum_{j}\partial_{y_{i}}a_{ji}\cdot\varphi_{j}(X_{t}(y))\\
= & \int_{\Omega}\chi_{0}(y)\text{Tr}\nabla_y (\widetilde{\varphi}(y))-\int_{\Omega}\chi_{0}(y)\sum_{i}\sum_{j}\partial_{y_{i}}a_{ji}\cdot\varphi_{j}(X_{t}(y)) \\
= & I_3 - I_4.
\end{align*}
Since $\chi_0\in {BV}(\Omega)$, we have
\begin{align*}
&\left|
I_3
\right|
=
\left|
\int_{\Omega}\chi_{0}(y){\rm div}_y \widetilde{\varphi}(t,y)
\right|
\leq 
\left\| \nabla\chi_{0} \right\|_{\mathcal{M}(\Omega) } 
\left\| \widetilde{\varphi} \right\|_{L^{\infty}([0,T]\times\Omega)}
\\
&
\leq 
C
\left\| \nabla\chi_{0} \right\|_{\mathcal{M}(\Omega) } 
\left\| \varphi \right\|_{L^{\infty}(\Omega)} 
 \left\| \nabla A \right\|_{L^{\infty}([0,T]\times\Omega)}
\\
&
\leq 
\left\| \chi_0 \right\|_{{BV}(\Omega) } 
\left\| \varphi \right\|_{L^{\infty}(\Omega)} 
\beta(\left\| u \right\|_{C([0,T];C^2(\overline{\Omega}))} ).
\end{align*}
The notation $\beta(\cdot)$ denotes a continuous function. 
%We will still use $\beta$ later but it may denote different continuous functions. 
In $\left\| \nabla A\right\|_{L^{\infty}(\Omega)}$, we firstly find the Euclidean norm $| \nabla A|$ and then find the $L^{\infty}$ norm of $| \nabla A|$. The situations later will be treated in the same way.  
We then estimate $I_4$:
\begin{align*}
&\left|
I_4
\right|
\leq 
C\left\| \nabla A \right\|_{L^{\infty}([0,T]\times\Omega)}
\left\| \chi_{0} \right\|_{L^1 (\Omega)} 
\left\| \varphi(X_{t}(y)) \right\|_{L^{\infty}(\Omega)}
\\
&
\leq
\left\| \chi_0 \right\|_{{BV}(\Omega) } 
\left\| \varphi \right\|_{L^{\infty}(\Omega)} 
\beta(\left\| u \right\|_{C([0,T];C^2(\overline{\Omega}))} ).
\end{align*}
We still use the notation $\beta(\cdot)$, so it represents
different continuous functions in different contexts.
The estimates of $I_3$ and $I_4$ implies
\begin{align*}
\left|
\int_{\Omega}\chi (X_{t}(y),t){\rm div}\varphi(X_{t}(y)){d}y
\right|
\leq 
\left\| \chi_0 \right\|_{{BV}(\Omega) } 
\left\| \varphi \right\|_{L^{\infty}(\Omega)} 
\beta(\left\| u \right\|_{C([0,T];C^2(\overline{\Omega}))} ).
\end{align*}
Thus, we have 
$$\mathcal{V}(\chi(t),\Omega)\leq \left\| \chi_0 \right\|_{{BV}(\Omega) } 
\beta(\left\| u \right\|_{C([0,T];C^2(\overline{\Omega}))}).$$
Noticing that $\left\| \chi(t)\right\| _{L^{1}(\Omega)}=\left\| \chi_{0}\right\| _{L^{1}(\Omega)}$
and $\left\| \nabla\chi(t)\right\| _{\mathcal{M}(\Omega)}\leq \mathcal{V}(\chi(t),\Omega)$, 
one has
\begin{align*}
& \left\| \chi(t)\right\| _{BV(\Omega)} 
 = \left\| \chi(t)\right\| _{L^1(\Omega)} + \left\| \nabla\chi(t)\right\| _{\mathcal{M}(\Omega)} \\
& \leq\left\| \chi_{0}\right\| _{L^{1}(\Omega)}+\mathcal{V}(\chi(t),\Omega)
\leq  \beta(\left\| u\right\| _{C([0,T];C^{2}(\overline{\Omega}))})\left\| \chi_{0}\right\| _{BV(\Omega)}.
\end{align*}
\begin{Remark}
In order to control $ (\nabla X_{t})^{-1}$ 
with $\beta(\left\| u \right\|_{C([0,T];C^2(\overline{\Omega}))})$, 
we only need to consider $\nabla X_{t} $.
This is because $\text{det}(\nabla X_{t}) =1$, which implies that $(\nabla X_t)^{-1}=\text{adj}(\nabla X_t)$.
We take the derivatives of the following equation: 
\begin{equation}
X(t,x)=x+\int_0^t u(s,X(s,x)) {d}s,
\end{equation}
and then we use the Gronwall's inequality.
In order to estimate $\partial_{y_{i}}a_{ji}$, 
we take the derivatives of the equation 
$ (\nabla X_{t})^{-1} = \text{adj}(\nabla X_{t}) $.
It remains to estimate the second derivatives of $X(t,x)$,
which can be solved similarly by the Gronwall's inequality.
\end{Remark}
\begin{Remark}
When  $\left\|u\right\|_{C([0,T]; C^2(\overline{\Omega}))}\leq R$
we have $\left\|\chi(u)\right\|_{L^{\infty}([0,T];BV(\Omega))}\leq C(R)$.
\end{Remark}

Now we study the operator $B(\cdot)$. We recall the Lemma 3.2 from \cite{hw}.
\begin{Lemma}[\cite{hw}, Lemma 3.2]\label{lemma Hu}  
Let $\Omega\subseteq\mathbb{R}^{3}$ be a bounded $C^3$-domain
and $u\in C([0,T];C_{0}^{2}(\overline{\Omega}))$.
There exists a unique solution operator $B(\cdot)$, such that 
$B(u)$ solves \eqref{eq:s,B}, \eqref{divB 0} and \eqref{eq:s,initial} in the weak sense. 
Given any bounded set $A\subseteq C([0,T];C_{0}^{2}(\overline{\Omega}))$,
the image $B(A)$ is bounded in $L^2([0,T];H_0^1(\Omega))\cap L^{\infty}([0,T];L^2(\Omega))$
and $B(\cdot)$ is continuous on $A$.
\end{Lemma}
We will show later that the condition $u\in C([0,T];C_0^2(\overline{\Omega}))$ 
will be guaranteed. 
Thus, we can always use the operator $B(\cdot)$ when solving the approximate equations. 
Note that when $\left\|u\right\|_{C([0,T];C^2(\overline{\Omega}))}\leq R$, 
we have 
$\left\|B(u)\right\|_{L^2([0,T];H_0^1(\Omega))} 
+ \left\|B(u)\right\|_{ L^{\infty}([0,T];L^2(\Omega)) } \leq C(R)$.

\subsection{Estimating the operator $N(u,\chi,B)$}
Substituting $\chi(u)$ and $B(u)$, 
we obtain $N(u)=N(u,\chi(u),B(u))$. 
%$N(u,\chi,B)=N(u)$. 
We estimate its operator norm.
For convenience we denote the formula by 
\begin{align*}
\left\langle N(u),\eta\right\rangle  = I_1 + I_2+I_3+I_4,
\end{align*}
where
\begin{align*}
&I_1 = \int_{\Omega}u\otimes u:\nabla\eta {d}x, \\
&I_2 = -\int_{\Omega}B\otimes B:\nabla\eta {d}x, \\
&I_3 = -2\int_{\Omega}\nu(\chi)Du:D\eta {d}x,  \\
&I_4 = \kappa \int_{\Omega}P_{\tau}:\nabla\eta{d}|\nabla\chi|.
\end{align*}
We estimate the integrals as follows.
Notice that the $C^1$ norm is equivalent to the $G_n$ norm
in the finite dimensional space $G_n$. Thus, we obtain
$$
|I_1|\leq\int_{\Omega}|u|^2|\nabla\eta|\leq\left\| \eta\right\| _{C^1(\Omega)}\left\| u\right\|_{L^2(\Omega)}^2 \leq C\left\| u\right\| _{G_{n}}^2 \left\| \eta\right\| _{G_n}.
$$
Similarly, we estimate $I_2$ and $I_3$ as
$$
|I_2|\leq\left\| \eta\right\| _{C^1(\Omega)}\left\| B\right\| _{L^2(\Omega)}^{2}\leq C\left\| B\right\|_{L^2(\Omega)}^{2}\left\| \eta\right\| _{G_{n}},
$$
$$
|I_3|
\leq C\left\| Du\right\| _{L^2(\Omega)}\left\| D\eta\right\| _{L^2(\Omega)}\\
\leq C\left\| u\right\| _{G_{n}}\left\| \eta\right\| _{G_{n}}.
$$
Using the fact that $|(P_{\tau})_{ij}|=|\delta_{ij}-n_i n_j|\leq 1$, 
we have
$$
\left|P_{\tau}:\nabla\eta\right|\leq\left|P_{\tau}\right|\left|\nabla\eta\right|\leq C\left|\nabla\eta\right|\leq C\left\| \eta\right\| _{G_{n}},
$$
and then we obtain
$$
|I_4| \leq
\left\| \chi\right\| _{BV(\Omega)}
\left\| P_{\tau}:\nabla\eta\right\| _{L^{\infty}(\Omega)}\leq C\left\| \chi\right\| _{BV(\Omega)}\left\| \eta\right\| _{G_{n}}.
$$
Thus, the operator norm of $N(u)$ is estimated as the following:
\begin{equation}\label{estimate functional N}
\left\| N(u)\right\| _{G_{n}^{*}}(t)\leq C\left(\left\| u\right\| _{G_{n}}^{2}+\left\| u\right\| _{G_{n}}+\left\| B(u)\right\| _{L^2(\Omega)}^{2}+\left\| \chi(u)\right\| _{BV(\Omega)}\right)(t).
\end{equation}

\subsection{Properties of the iterating operator}
In order to construct the iterating operator, 
we consider the equation
\begin{equation}
 M(u(t)) = M(u_{0}) +\int_0^t N(u(s)){d}s,
\end{equation}
which can be rewritten as:
\begin{equation}
 u(t) = M^{-1}( M(u_{0})) +M^{-1}\left( \int_0^t N(u(s)){d}s \right) .
\end{equation}
In order to prove $M$ is invertible,
we suppose $M(\eta)=0$ in $G_n^*$, then we have $\left\| \eta \right\|_{L^2(\Omega)}^2=\langle M(\eta),\eta \rangle=0$. 
Since $\eta$ is a continuous function, 
we have $\eta=0$. Thus, $M:G_n\to G_n^*$ is invertible. 
\begin{Remark}
When $u_0\in L^2 (\Omega)$, the functional $M(u_0)\in G_n^*$ is still well defined. 
\end{Remark}
Let $\widetilde{u}_0:=M^{-1}(M(u))$, we define
\begin{equation}
K(u(t)) := M^{-1}\left(M(u_{0}) + \int_{0}^{t}N(u(s)) {d}s \right)
=\widetilde{u}_0 +M^{-1}\left( \int_{0}^{t}N(u(s)) {d}s \right).
\end{equation}
For convenience, we define the set:
$$A_{a,b}:=\{u\in C([0,a],G_{n}):\left\| u\right\| _{L^{\infty}([0,a];G_{n})}\leq b\}.$$
\begin{Remark}
Since all norms are equivalent in $G_{n}$,
we pick an arbitrary norm and fix it to be our $\left\| \cdot\right\| _{G_{n}}$.
We will consider the properties 
of $K(\cdot)$ on $A_{T^*,R}$.
In fact, $K(\cdot)$ becomes a compact operator on $A_{T^*,R}$ for suitable $T^*$ and $R$.
\end{Remark}
Given $u\in A_{T,R}$. 
From \eqref{estimate functional N} 
and Section \ref{solu operators},
 we have the following estimate: 
\begin{equation}\label{estimate operator N time}
\left\| N(u)\right\| _{L^{\infty}([0,T];G_n^*)}
\leq 
C(R).
\end{equation}
We now study the properties of $K$. 
Firstly, we study the continuity of $K$ 
with respect to $t$. In fact, from \eqref{estimate operator N time},
 we have 
\begin{equation}\label{continuity of Ku in t}
\begin{aligned}
 & \left\| K(u)(t)-K(u)(s)\right\| _{G_n} \\
& \leq
\left\| M^{-1}\right\| _{\mathcal{L}(G_n^*,G_n)} 
\int_s^t\left\| N(u)\right\|_{G_n^*} (r) {d}r   \\
& \leq C(R)|t-s|.
\end{aligned}
\end{equation}
Thus, $K(u)\in C([0,T],G_{n})$.

Secondly, we study the boundedness of $K(u)$.
Still using the estimate in \eqref{estimate operator N time}, we obtain
\begin{equation} \label{estimation of K}
\left\| K(u)\right\| _{G_{n}}(t)\leq\left\| \widetilde{u}_0\right\| _{G_{n}}+C(R)t.
\end{equation}
We choose $R>\left\| \widetilde{u}_0 \right\| _{G_{n}}$
and $T^{*}$ small
enough, such that
$$
\left\| \widetilde{u}_0\right\| _{G_{n}}+C(R)T^{*} < R.
$$
Then the operator $K$ maps $A_{T^{*},R}$ into $A_{T^{*},R}$. 

Thirdly, we show that $K(\cdot)$ is a continuous operator
on $A_{T^{*},R}$. 
We fix $v\in A_{T^{*},R}$ and let 
 $u\in A_{T^{*},R}$ be such that $u\to v$ in $C([0,T^*];G_n)$.
Since
\begin{align*}
\left\| K(u)-K(v)\right\|_{G_{n}}(t) \leq\left\| M^{-1}\right\| _{\mathcal{L}(G_{n}^{*},G_{n})}\int_{0}^{t}\left\| N(u)-N(v)\right\| _{G_n^*}(s){d}s,
\end{align*}
we need to estimate $\left\| N(u)-N(v)\right\| _{G_n^*}(t)$. 
For any $t\in[0,T^{*}]$, we consider 
\begin{equation}\label{NuNv estimate}
\begin{aligned}
 \left<N(u)-N(v),\eta\right>
 =I_1+I_2+I_3+I_4\\
\end{aligned},
\end{equation}
where
\begin{equation}
\begin{aligned}
&
I_1 =
\int_{\Omega}(u\otimes u-v\otimes v):\nabla\eta, \\
& I_2 =
-\int_{\Omega}(B_{u}\otimes B_{u}-B_{v}\otimes B_{v}):\nabla\eta, \\
& I_3 = -2\int_{\Omega}\left(\nu(\chi_u)Du-\nu(\chi_v)Dv\right):\nabla\eta,  \\
& I_4 = \kappa\left( \int_{\Omega}P_{\tau}:\nabla\eta{d}
\left|\nabla\chi_{u}\right| -
\int_{\Omega}P_{\tau}:\nabla\eta{d}\left|\nabla\chi_{v}\right|\right).
\end{aligned}
\end{equation}
Here we denote $B(u)$ and $B(v)$  by $B_u$ and $B_v$ for short.  
The variable $t$ is ignored for convenience when there is no ambiguity.
The terms $I_1$ to $I_4$ are estimated as follows. 
For $I_1$, we have
\begin{align*}
|I_1| & \leq \int_{\Omega}\left|u\otimes(u-v)+(u-v)\otimes v\right||\nabla\eta|{d}x
\\
\leq & \int_{\Omega}|u||u-v||\nabla\eta| {d}x
+
\int_{\Omega}|u-v||v||\nabla\eta| {d}x
\\
\leq & \left\| u\right\|_{L^2(\Omega)} \left\| u-v\right\| _{L^2(\Omega)}
\left\| \nabla\eta\right\| _{L^{\infty}(\Omega)}
+ \left\| u-v\right\|_{L^2(\Omega)}  \left\| v\right\|_{L^2(\Omega)} 
\left\| \nabla\eta\right\| _{L^{\infty}(\Omega)} 
\\
\leq & C\left( \left\| u\right\| _{G_{n}} + \left\| v\right\| _{G_{n}} \right) \left\| u-v\right\| _{G_{n}}\left\| \eta\right\| _{G_{n}} \\
\leq & CR\left\| u-v\right\| _{G_{n}}\left\| \eta\right\| _{G_{n}}.
\end{align*}
Thus,
$\sup_{[0,T^*]}|I_1(t)|\leq CR\left\| u-v\right\| _{C([0,T^*];G_n)}\left\| \eta\right\| _{G_{n}}$.
Similarly to the estimate of $I_1$, we obtain
\begin{align*}
|I_2| & \leq\int_{\Omega}|B_u||B_{u}-B_v||\nabla\eta|+\int_{\Omega}|B_{u}-B_{v}||B_{v}||\nabla\eta|\\
\leq & C\left(\left\| B_u\right\| _{L^2}+\left\| B_v\right\| _{L^2}\right)\left\| B_u-B_v\right\| _{L^2}\left\| \eta\right\| _{G_{n}}\\
\leq & C(R)\left\| B_u-B_v\right\| _{L^2}\left\| \eta\right\| _{G_{n}}.
\end{align*}
From Lemma \ref{lemma Hu}, we have
$\sup_{[0,T^*]}|I_{2}(t)|
\leq 
C(\left\| u-v\right\| _{C([0,T^*];G_n)})\left\| \eta\right\| _{G_n}$.
Moreover, the constant $C(\left\| u-v\right\| _{C([0,T^*];G_n)})\to 0$ 
when $\left\| u-v\right\| _{C([0,T^*];G_n)}\to 0$. 
For $I_3$ we obtain
\begin{align*}
|I_3| & \leq  2\int_{\Omega}\left|\nu(\chi_u)Du-\nu(\chi_v)Dv\right||\nabla\eta|\\
\leq & 2\int_{\Omega}\left|\nu(\chi_u)\right||Du-Dv||\nabla\eta|
+
2 \int_{\Omega}\left|\nu(\chi_u)-\nu(\chi_v)\right||Dv||\nabla\eta|\\
\leq & C\int_{\Omega}|\nabla(u-v)||\nabla\eta|+C\int_{\Omega}\left|\nu^{+}\chi_{u}+\nu^{-}(1-\chi_{u})-\nu^{+}\chi_{v}-\nu^{-}(1-\chi_{v})\right||\nabla v||\nabla\eta|\\
\leq & C\left\| u-v\right\| _{G_{n}}\left\| \eta\right\| _{G_{n}}+C\left\| v\right\| _{G_{n}}\left\| \eta\right\| _{G_{n}}\int_{\Omega}\left|\chi_{u}-\chi_{v}\right|.
\end{align*}
The key point is to prove that $\int_{\Omega}\left|\chi_u-\chi_v\right|\to 0$
as $u\to v$ in $C([0,T^*];G_n)$.
In fact, $u\to v$ in $C([0,T^*];G_n)$ implies $u\to v$ in $C([0,T^*];C^{1}(\overline{\Omega}))$. 
Thus, similarly to the argument in \cite{Abels1}, 
we obtain $X_u\to X_v$
in $C([0,T^*];C^1(\Omega))$.
Let
$\Omega^+_u(t):=X_u(t,\Omega^+_0)$,
$\Omega^+_v(t):=X_v(t,\Omega^+_0)$
and
$\Gamma_u(t):=X_u(t,\Gamma_0)$, $\Gamma_v(t):=X_v(t,\Gamma_0)$.
Notice that
\begin{align*}
\int_{\Omega} & \left|\chi_u-\chi_v\right| {d}x
=
\left|\Omega^+_u \triangle \Omega^+_v \right|
\end{align*}
where $\triangle$ denotes the symmetric difference of sets.
For any $\varepsilon>0$, 
if $\left\| X_u - X_v \right\|_{C([0,T^*];C^1(\overline{\Omega}))}<\varepsilon$,
then
$\Gamma_u(t)\subseteq B(\Gamma_v(t),\varepsilon)$ for every $t\in[0,T^*]$.
Here
$B(\Gamma_v(t),\varepsilon)$ is the $\varepsilon$-neighbourhood of $\Gamma_v(t)$.
Since $v\in C([0,T^*];G_n)\subseteq C([0,T^*];C^2(\Omega))$, 
we obtain that the flow mapping $X_v(t,x)\in C([0,T^*];C^2(\Omega))$.
Since our $\Gamma_0$ is a $C^2$-surface, 
we can apply a local parameterization to $\Gamma_0$.
By composing with $X_v(t,x)$, it will naturally give us a local parameterization of $\Gamma_v(t)$. 
Suppose that $\varphi(a_1,a_2)$ 
is a $C^2$-diffeomorphism from an open set $D\subseteq \mathbb{R}^2$
to a local piece of $\Gamma_v(t)$.
Using the normal vector $n(\varphi(a_1,a_2))$,
the function
$$\psi(a_1,a_2,a_3):=\varphi(a_1,a_2)+ a_3 n(\varphi(a_1,a_2))$$
gives us a diffeomorphism from $D\times(-\varepsilon,\varepsilon)$
to an open set in $B(\Gamma_v(t),\varepsilon)$.
This allows us to 
obtain a local parameterization of $B(\Gamma_v(t),\varepsilon)$.
Notice that both $D\times(-\varepsilon,\varepsilon)$ 
and $\psi(D\times(-\varepsilon,\varepsilon))$
are monotone increasing sets as $\varepsilon$ increases.
Thus, we can obtain the boundedness of the integrands 
and then use the Lebesgue dominated convergence theorem.
When $\varepsilon\to 0$, by calculating the integrals, 
we have $|B(\Gamma_v(t),\varepsilon)|\to 0$
uniformly in $t$.
Thus, $\left\| u-v\right\| _{C([0,T^*]; G_{n})}\to 0$ implies
$$
\sup_{t\in [0,T^*]}\int_{\Omega}  \left|\chi_{u}-\chi_{v}\right|(t){d}x \to 0.
$$
Hence, $\sup_{[0,T^*]}|I_3(t)|\leq C(\left\| u-v\right\| _{C([0,T^*]; G_{n})})\left\| \eta\right\| _{G_{n}}$. 
Similarly to the constant term in $I_2$, 
the constant $C(\left\| u-v\right\| _{C([0,T^*]; G_{n})})\to 0$ 
as $\left\| u-v\right\| _{C([0,T^*];G_n)} \to 0$.

In order to estimate $I_4$, we consider the functional $F_u(t)$ such that
$$
\langle F_u(t),\eta\rangle
:=
\left\langle H_{\chi_u(t)},\eta \right\rangle -\left\langle H_{\chi_v (t)},\eta \right\rangle.
$$
Thus, $I_4(t)=\kappa\langle F_u(t), \eta \rangle$.
Suppose $\left\| u-v\right\|_{C([0,T^*]; G_n)}\to 0$,
we need to prove that 
$$
\left\| F_u\right\|_{L^{\infty}([0,T^*];G_n^*)} \to 0.
$$
Note that $\{\eta\in G_n:\left\| \eta\right\| _{G_n}= 1\}$
is a subset of $A:=\{\eta\in G_n:\left\| \eta\right\| _{C^2(\Omega)}\leq C\}$
for a suitable $C$.
Thus, it is sufficient to show 
\begin{equation}\label{Fu goes to 0}
\sup_{t\in[0,T^*]} \sup_{\eta\in A} \left| \langle F_u(t),\eta\rangle \right| \to 0.
\end{equation}
Note that since $u$, $v\in C([0,T^*],G_n)$, 
 the interfaces $\Gamma_u(t)$ and $\Gamma_v(t)$
are both $C^2$-surfaces for all $t\in[0,T^*]$.
Since $\Gamma_0$ is compact,
by applying a local parameterization and using the partition of unity,
 we can consider the integrals locally.
Let $\varphi(a_1,a_2)$ be the $C^2$-diffeomorphism from an open set $D\subseteq\mathbb{R}^2$
to a local piece on $\Gamma_0$.
The function $X_u(t,\varphi(a_1,a_2))$ allows us 
to calculate the normal vector 
$$n_u(X_u(t,\varphi(a_1,a_2)))\in C([0,T^*];C^1(D)).$$
When $u\to v$ in $C([0,T^*];C^1(\Omega))$, 
we have $X_u\to X_v$ in $C([0,T^*];C^1(\Omega))$.
Thus, $n_u\to n_v$ in $C([0,T^*]\times D)$.
Similarly, the Jacobians $J(X_u(t,\varphi(a_1,a_2)))$ 
goes to $J(X_v(t,\varphi(a_1,a_2)))$ in $C([0,T^*]\times D)$,
and the test functions $\nabla\eta(X_u(t,\varphi(a_1,a_2)))$
goes to $\nabla\eta(X_v(t,\varphi(a_1,a_2)))$ in $C([0,T^*]\times D)$ as well.
Then \eqref{Fu goes to 0} is obtained by calculating the integrals.

%Similarly to the argument in \cite{Abels1}, Theorem 4.2, 

Thus, for all $\eta\in G_n$, we have
\begin{align*}
|I_4|
=
\kappa\left| \left\langle F_u(t) , \eta \right\rangle \right|
\leq
C(\left\|u-v \right\|_{C([0,T^*];G_n)})\left\| \eta  \right\|_{G_n}.
\end{align*}
The constant $C(\left\|u-v \right\|_{C([0,T^*];G_n)})$ goes to $0$ 
as $u\to v $ in $C([0,T^*];G_n)$.
From the estimates of $I_1$ to $I_4$, we have 
\begin{equation}
\left\| N(u)-N(v)\right\|_{C([0,T^*];G_n^*)}\leq
C(\left\| u-v \right\|_{C([0,T^*];G_n)}) \to 0.
\end{equation}
Finally, we obtain
\begin{align*}
&\left\| K(u)-K(v)\right\| _{C([0,T^*];G_n)} \\
& \leq\left\| M^{-1}\right\| _{\mathcal{L}(G_{n}^{*},G_{n})}T^{*}
\left\| N(u)-N(v)\right\| _{C([0,T^*];G_n^*)}\\
& \leq  C(\left\| u-v\right\| _{C([0,T^*];G_n)})\to 0,
\end{align*}
which implies that $K(\cdot)$ is continuous on $A_{T^{*},R}$.

Now we prove that $K(A_{T^{*},R})$ is precompact. 
%Using Theorem \ref{Arzela Ascoli}.
Given any $v\in K(A_{T^*,R})$ and $t\in[0,T^*]$, 
there exists $u\in A_{T^*,R}$ such that $v=Ku$.
From \eqref{estimation of K}, we have
$$
\left\| v\right\| _{G_{n}}(t)=\left\| Ku \right\| _{G_{n}}(t)\leq\left\| \tilde{u}_{0}\right\| _{G_{n}}+C(R)T^*.
$$
Thus, the set $\{v(t):v\in K(A_{T^{*},R})\}\subseteq G_{n}$ is precompact since it is bounded and $G_{n}$ is finite dimensional.  
From (\ref{continuity of Ku in t}), 
the functions in $K(A_{T^{*},R})$ are equicontinuous.
Thus, from Proposition \ref{Zeidler1 Appendix 24g}, $K(A_{T^{*},R})$ is precompact. 
Since $A_{T^{*},R}$ is already bounded, the operator $K$ maps all the bounded subsets of $A_{T^{*},R}$ 
into precompact sets. 
Thus, from Definition \ref{def compact operator}, we obtain that $K$ is a compact operator.

It remains to verify the properties of the set $A_{T^{*},R}$ in $C([0,T^*];G_n)$.
Since $u(t)\equiv0$ is in $A_{T^{*},R}$, the set is non-empty. 
From the definition of $A_{T^{*},R}$, 
we know it is closed and bounded. 
For the convexity of $A_{T^*,R}$, picking any $u$ and $v$
in $A_{T^{*},R}$ and any $0\leq\theta\leq 1$, we have
$$
\begin{aligned}\left\| \theta u+(1-\theta)v\right\|_{G_{n}}(t)\leq\theta\left\| u\right\| _{G_{n}}(t)+(1-\theta)\left\| v\right\| _{G_{n}}(t) \leq R
\end{aligned}.
$$
Thus, $A_{T^{*},R}$ is convex.

From Theorem \ref{fixed point thm},  there exists a solution $u_{n}(t)\in C([0,T^{*}];G_{n})$.
Replacing the initial value $u_0$ by  $u(T^*)$
and repeating the steps above, 
we can increase the value of $T^*$. 
Currently we can only guarantee that 
there will be a limit when we increase $T^*$. 
Thus, the maximum interval would be either $[0,T^*)$ or $[0,T]$, 
where $T^*\leq T$.
When $T^*$ is being excluded from the interval, 
it actually means the solution will goes to infinity when $t$ is approaching $T^*$.
This will not happen in our problem, as shown in the following section.

\subsection{Extending the solution to $[0,T]$}
Assuming that $T^{*}<T$, we derive a contradiction using the energy estimate.
For each fixed $n$, we need to prove that
$\sup_{[0,T]}\left\| u_{n}\right\| _{G_{n}}\leq C$,
which is equivalent to
$\sup_{[0,T]}\left\| u_{n}\right\| _{L^2}\leq C$.
Since $u_n$ is the solution of the approximate equation,
we take the derivative of (\ref{eq: Galerkin u}) with respect to the variable $t$. 
Substituting $\eta$ with $u_{n}(t)$, 
and using (2.8) in \cite{Abels1},
we obtain
\begin{equation}
\begin{aligned}
\frac{1}{2}\frac{{d}}{{d}t}\left\| u_{n}\right\| _{L^2}^{2} 
+\kappa \frac{d}{dt}\left\| \nabla\chi_{n}\right\| _{\mathcal{M}(\Omega)}
 +(B_{n}\otimes B_{n},\nabla u_{n})_{\Omega}+2(\nu(\chi_n)Du_{n},Du_{n})_{\Omega}
%+ & \kappa\int_{\Omega}P_{\tau}:\nabla u_{n} {d}\left|\nabla\chi_{n}(t)\right|
=0.
\end{aligned}
\end{equation}
Integrating from $0$ to $t$, we have
\begin{equation} \label{Xn energy u}
\begin{aligned}
 & \frac{1}{2}\left\| u_{n}(t)\right\| _{L^2}^{2} 
+\kappa\norm{\nabla\chi_{n}(t)}_{\mathcal{M}(\Omega)}
 +\int_0^t(B_{n}\otimes B_{n},\nabla u_{n})_{\Omega}{d}s
 \\
 &+
 2\int_0^t(\nu(\chi_n)Du_{n},Du_{n})_{\Omega}{d}s 
=
\frac{1}{2}\left\| u_0\right\| _{L^2}^{2}
+
\kappa\norm{\nabla\chi_0}_{\mathcal{M}(\Omega)}.
\end{aligned}
\end{equation}
For each $u_n$, the solution operator $B(\cdot)$ gives us a weak solution of
\eqref{eq:s,B}.
Thus, by testing \eqref{eq:s,B} with $\varphi=B_{n}$ on $\Omega\times[0,t]$,
we obtain
\begin{equation} \label{Xn energy B}
\begin{aligned}
\frac{1}{2}\left\| B_{n}(t)\right\| _{L^2(\Omega)}^{2}
-
\frac{1}{2}\left\| B_0\right\| _{L^2(\Omega)}^{2}
-
\int_0^t (B_{n}\otimes B_{n},\nabla u_{n})_{\Omega}{d}s
+
\sigma\int_0^t \left\| \nabla B_{n}\right\| _{L^2(\Omega)}^{2}=0.
\end{aligned}
\end{equation}
Using the same argument as in the energy estimate, for some $c>0$, we have
\begin{equation}
\begin{aligned}
&
\frac{1}{2}\left\| u_n (t)\right\| _{L^2(\Omega)}^{2}+ 
\frac{1}{2}\left\| B_n (t)\right\| _{L^2(\Omega)}^{2} +\kappa\left\| \nabla\chi_n (t)\right\| _{\mathcal{M}(\Omega)}
+
c\int_0^t \left\| \nabla u_n(s) \right\| _{L^2(\Omega)}^{2}{d}s 
\\
&
+
\sigma\int_0^t \left\| \nabla B_n(s) \right\| _{L^2(\Omega)}^{2}{d}s
\leq\frac{1}{2}\left\| u_{0}\right\| _{L^2(\Omega)}^{2}+\frac{1}{2}\left\| B_{0}\right\| _{L^2(\Omega)}^{2}+\kappa\left\| \nabla\chi_{0}\right\| _{\mathcal{M}(\Omega)}=E_{0}
\end{aligned}
\end{equation}
for any $t\in[0,T^*]$. 
Thus, $\sup_{[0,T^*]}\left\| u_{n}(t)\right\| _{G_{n}}\leq\sup_{[0,T^*]} C \left\| u_{n}(t)\right\| _{L^2(\Omega)}\leq  C $.

From the ODE theory, we know that if the maximum interval of a solution
is
$[0,T^{*})$, then the solution must blow up at $T^*$. 
We will use the same argument. 
Picking an increasing sequence $t_{m}\in[0,T^*)$ such that $t_{m}\to T^*$,
we consider the sequence $\{u_n (t_{m})\}_{m=1}^{\infty}\subset G_{n}$.
Since $\sup_{[0,T]}\left\| u_n \right\| _{G_{n}}(t)\leq C$ and
$\text{dim} G_{n}<\infty$, 
we can find a subsequence, still denoted by $ t_m $, such that 
$u_n (t_m)\to a\in G_{n}$,
as $m\to\infty$. 
We only need to prove that $\lim_{t\to T^*}\left\| u_n (t)-a \right\|_{G_n} =0$.
Then the solution $u_n (t)$ can be continuously extended to $[0,T^*]$.
Using the Schauder fixed point theorem again, with $T^*$ being 
the new initial time, we will get a contradiction.
It then follows that $T^*=T$. 
Now we assume that $\lim_{t\to T^*}u_n (t)\neq a$,
then there exists an $\varepsilon_{0}>0$,
such that for all $\delta>0$, there exists $T^*-\delta< s <T^*$,
such that $\left\| u_n (s)-a\right\| _{G_{n}}>\varepsilon_{0}$.
Meanwhile, there exists an $m$,
 such that  $T^{*}-\delta<t_{m}<T^{*}$
and $\left\| u_n (t_m)-a\right\| _{G_{n}}<\varepsilon_{0}/2$.
Thus, we obtain
$$\left\| u_n (s)-u_n (t_m)\right\| _{G_{n}}
\geq
\left\| u_n (s)-a\right\| _{G_n}-\left\| u_n (t_m)-a\right\| _{G_n}>\varepsilon_{0}/2.$$
Recall that
\begin{align*}
& \left\| u_n (s)-u_n (t_m)\right\| _{G_{n}}
=
\left\| M^{-1}\int_{t_m}^{s}N(u_n )\right\| _{G_n}\leq\left\| M^{-1}\right\| 
\int_{t_m}^{s}\left\| N(u_n )\right\| _{G_n^*}\\
%& \leq C|s - t_m|\left(\left\| u_n\right\| _{L^{\infty}([0,T];G_{n})}^{2}
%+
%\left\| u_n\right\| _{L^{\infty}([0,T];G_{n})}
%+
%\left\| B(u_n)\right\| _{L^{\infty}([0,T];L^2(\Omega))}^{2}
%+
%\left\| \chi(u)\right\| _{L^{\infty}([0,T];BV(\Omega))}\right)\\
& \leq C|s-t_m|<C\delta.
\end{align*}
Let $\delta$ be small enough such that $C\delta<\varepsilon_{0}/2$, 
then we get a contradiction. 
Thus, $\left\| u_n (t)- a \right\|_{G_n} \to 0$ as $t\to T^*$.

Consequently, we have found a solution $u_n\in C([0,T];G_{n})$.
Using the solution operators, we obtain the corresponding $B_n:=B(u_n)$ and
$\chi_n:=\chi(u_n )$. 
The energy inequality
\begin{equation}\label{ineq: Galerkin energy estimate}
\begin{aligned}
\frac{1}{2}\left\| u_n (t)\right\| _{L^2(\Omega)}^{2}+\frac{1}{2}\left\| B_n (t)\right\| _{L^2(\Omega)}^{2}+\kappa\left\| \nabla\chi_n (t)\right\| _{\mathcal{M}(\Omega)}\\
+c\left\| \nabla u_n \right\| _{L^2([0,T];L^2(\Omega))}^{2}+\sigma\left\| \nabla B_n \right\| _{L^2([0,T];L^2(\Omega))}^{2}\leq E_{0}
\end{aligned}
\end{equation}
holds for all $t\in[0,T]$.

\bigskip

\section{Passing the Limit}\label{pass limit}

In this section, we study the limits of $u_n$, $B_n$ and $\chi_n$.
Recall that $$u_n\in C([0,T];G_n),\quad 
B_n\in L^2([0,T];H^1_0(\Omega))\cap L^{\infty}([0,T];L^2(\Omega)),\quad
 \chi_n\in L^{\infty}([0,T];BV(\Omega)),$$
and ${\rm div}u_n={\rm div}B_n=0$.
From the enrgy inequality \eqref{ineq: Galerkin energy estimate}, 
we have the following estimates:
\begin{equation}\label{ineq: many estimates}
\begin{aligned}
\left\Vert u_{n}\right\| _{L^{\infty}([0,T];L^2(\Omega))} & \leq\sqrt{2E_{0}}, \\
\left\| u_{n}\right\| _{L^2([0,T];H_0^1(\Omega))} & \leq\sqrt{2TE_{0}+E_{0}/c}, \\
\left\Vert B_{n}\right\| _{L^{\infty}([0,T];L^2(\Omega))} & \leq\sqrt{2E_{0}}, \\
\left\| B_{n}\right\| _{L^2([0,T];H_0^1(\Omega))} & \leq\sqrt{2TE_{0}+E_{0}/\sigma}, \\
\left\Vert \nabla\chi_{n}\right\| _{L^{\infty}([0,T];\mathcal{M}(\Omega))} & \leq E_{0}/\kappa,  \\
\left\| \chi_{n}\right\| _{L^{\infty}([0,T];BV(\Omega))} & \leq\left|\Omega\right|+E_{0}/\kappa.
\end{aligned}
\end{equation}

\subsection{Limits of $u_{n}$, $B_{n}$ and $\chi_{n}$}

From the embedding theorems, we have
$$
\mathcal{M}(\Omega) \hookrightarrow H^{-3}(\Omega).
$$
From the estimates in \eqref{ineq: many estimates} and the Banach-Alaoglu theorem
 (see \cite{Brezis}), 
we have
\begin{align*}
u_{n}\rightharpoonup^{*}u & \ \ \ \text{in}\,\,L^{\infty}([0,T];L^2(\Omega)),\\
u_{n}\rightharpoonup v & \ \ \ \text{in}\,\,L^2([0,T];H_0^1(\Omega)),\\
B_{n}\rightharpoonup^{*}B & \ \ \ \text{in}\,\,L^{\infty}([0,T];L^2(\Omega)),\\
B_{n}\rightharpoonup G &  \ \ \ \text{in}\,\,L^2([0,T];H_0^1(\Omega)),\\
\chi_{n}\rightharpoonup^{*}\chi &  \ \ \ \text{in}\,\,L^{\infty}([0,T];L^{\infty}(\Omega)),\\
\nabla\chi_{n}\rightharpoonup^{*}\zeta &  \ \ \ \text{in}\ \ L^{\infty}([0,T];H^{-3}(\Omega)),
\end{align*}
for suitable subsequences.

In order to pass the limit in nonlinear terms, we need to obtain stronger convergence properties of $u_n$.
We begin by showing an improved version of Lemma A.3 in \cite{JiangSong}.
\begin{Lemma} \label{C0Vdual.covergence}
Let $\Omega\subseteq \mathbb{R}^3$ be bounded. 
Suppose that $u_n\rightharpoonup^{*} u$  in $ L^{\infty}([0,T];L^2(\Omega))$,
and  for any $\varphi\in \mathbb{H}(\Omega)$,
\begin{equation} \label{weak.topo.convergence.u}
\sup_{t\in[0,T]} \abs{ \int_{\Omega} u_n \varphi dx - \int_{\Omega} u \varphi dx } \to 0.
\end{equation}
Then we have $u_n \to u$ in $C([0,T]; \mathbb{V}^*(\Omega))$.
\end{Lemma}

\begin{proof}
Assume that $u_n$ do not converge to $u$ in $C^0([0,T]; \mathbb{V}^*)$, 
then there exists $\varepsilon_0 > 0$ and $t_n\in[0,T]$,
 such that
\begin{equation}
\norm{ u_n  - u}_{\mathbb{V}^*} (t_n) 
> \varepsilon_0.
\end{equation}
Thus, there exist $\varphi_n\in \mathbb{V}(\Omega)$ with 
$\left\| \varphi_n\right\|_{ \mathbb{V} } \equiv 1$,
such that 
\begin{equation} \label{dualpair.lowerbound}
\abs{ \dualpair{u_n,\varphi_n}(t_n) - \dualpair{u,\varphi_n}(t_n) } > \frac{\varepsilon_0}{2},
\end{equation}
where $\dualpair{\cdot, \cdot}$ denotes the dual pair on a space and its dual. 
Since $\varphi_n\in \mathbb{V}= \mathbb{H} \cap H_0^1$ 
and $H_0^1 \hookrightarrow \hookrightarrow L^2$,
there exists $\varphi\in L^2(\Omega)$ 
such that $\varphi_n\to \varphi$ in $L^2(\Omega)$.
Using the fact that $\norm{\varphi_n}_{H_0^1(\Omega)}\leq C$, 
we can obtain $\varphi\in \mathbb{V}(\Omega)$.
Now we have
\begin{equation}
\sup_{t\in[0,T] } \abs{ \dualpair{ u_n - u, \varphi_n }}
\leq
\sup_{t\in[0,T] } \abs{ \dualpair{ u_n - u, \varphi_n -\varphi }}
+
\sup_{t\in[0,T] } \abs{ \dualpair{ u_n - u, \varphi }}.
\end{equation}
The first term on the right hand side goes to 0 since $u_n$ and $u$ are bounded
in $L^{\infty}([0,T];L^2(\Omega))$, and $\varphi_n \to \varphi$ in $L^2(\Omega)$; the second term goes to 0 
by the condition \eqref{weak.topo.convergence.u} of this lemma. This  contradicts with \eqref{dualpair.lowerbound}, 
which completes the proof.
\end{proof}

In order to use the lemma above,
we still need to verify \eqref{weak.topo.convergence.u}.
We recall that elements $\eta_i$ 
form an orthonormal basis of $\mathbb{H}(\Omega)$ 
and an orthogonal basis of $\mathbb{V}(\Omega)$.
Let $\eta$ be a finite linear combination of $\eta_i$.
For all sufficiently large $n$,
the Galerkin approximate equations
\begin{equation} \label{Galerkin.equation}
\begin{aligned}
&\int_{\Omega} u_n (t) \eta - \int_{\Omega} u_0 \eta
=
\int_0^t \int_{\Omega} u_n\otimes u_n : \nabla \eta 
- \int_0^t \int_{\Omega} B_n\otimes B_n : \nabla \eta 
\\
&
- \int_0^t \int_{\Omega} \nu(\chi_n) Du_n : D \eta
- \kappa\int_0^t \int_{\Omega} P_{\tau}  : \nabla \eta d\abs{\nabla \chi_n}.
\end{aligned}
\end{equation}
all hold for this $\eta$.
Considering the terms 
$f_n(t):= \int_{\Omega} u_n (t)\eta dx$,
we claim that $f_n$ have a uniformly convergent subsequence. 
%using Arzela-Ascoli theorem.
First, we proof that $f_n(t)$ are equicontinuous.
Given $0\leq s< t \leq T$. 
Since $\eta$ is fixed, from \eqref{Galerkin.equation} we have
\begin{equation} \label{proving.equicontinuous}
\begin{aligned}
&
\abs{f_n (t) - f_n (s)  }
\leq
C\int_s^t \norm{u_n}_{L^2}^2  
 +\norm{B_n}_{L^2}^2
+ \norm{\nabla u_n}_{L^1}
+ \norm{\nabla \chi_n}_{\mathcal{M}}
\\
&
%=: C\int_s^t \alpha_n
\leq 
\left(  \norm{u_n}_{L^{\infty} L^2}^2
+ \norm{B_n}_{L^{\infty} L^2}^2
+ \norm{\nabla \chi_n}_{L^{\infty}\mathcal{M}}
\right) \abs{t-s}
+\int_0^T \chi_{[s,t]}\norm{\nabla u_n}_{L^1}
\\
&
\leq
C \abs{t-s} + C \sqrt{t-s},
\end{aligned}
\end{equation}
%From the energy inequality in the previous section \Tian{under editing}, 
%we know that $\alpha_n$ is bounded in $L^2[0,T]$.
%Thus, by Hölder's inequality we have 
%\begin{equation}
%\begin{aligned}
%\int_s^t \alpha_n
%\leq 
%\int_0^T \chi_{[s,t]} \alpha_n
%\leq 
%\sqrt{t-s} \norm{\alpha_n}_{L^2[0,T]},
%\end{aligned}
%\end{equation}
which implies that $f_n$ is equicontinuous on $[0,T]$.
By letting $s=0$ we can show that $f_n$ is uniformly bounded.
Thus, by the  Arzela-Ascoli theorem, there exists a subsequence, 
still denoted by $f_n$, such that
\begin{equation}
\sup_{t\in[0,T]}\abs{f_n(t) - g(t) } \to 0
\end{equation}
for some $g(t)\in C[0,T]$, as $n\to\infty$.

We recall that $u_n\rightharpoonup u$ in %$L^{\infty}([0,T];L^2(\Omega))$.
$L^2([0,T];H_0^1(\Omega))$.
Now we want to show that 
\begin{equation} \label{g.equals.integral}
g=\int_{\Omega} u\eta
\end{equation}
almost everywhere on $[0,T]$.
Given any $\varphi\in L^2[0,T]$.
Since $\eta\varphi\in L^2([0,T];L^2(\Omega))$ 
and $u_n\rightharpoonup^{*} u$  in $ L^{\infty}([0,T];L^2(\Omega))$,
 we have
\begin{equation}
\int_0^T f_n \varphi dt
=
\int_0^T \int_{\Omega} u_n \eta \varphi dxdt
\to
\int_0^T \int_{\Omega} u \eta \varphi dxdt,
\end{equation}
which implies that 
$f_n \rightharpoonup \int_{\Omega} u \eta$ weakly in $L^2[0,T]$.
Since $f_n\to g$ in $C[0,T]$, we also have
$f_n \rightharpoonup g$ weakly in $L^2[0,T]$, 
which proves \eqref{g.equals.integral}.

Now we prove that \eqref{weak.topo.convergence.u} holds for all $\varphi\in \mathbb{H}(\Omega)$.
Picking $\eta_1$, using the argument above we can find a subsequence of $u_n$, 
denoted by $u_{1n}$,
such that 
$
\int_{\Omega} u_{1n} \eta_1
\to
\int_{\Omega} u \eta_1
$ in $C[0,T]$.
Now picking $\eta_2$ and using the same argument, we can obtain a subsequence of $u_{1n}$, denoted by $u_{2n}$, such that
$
\int_{\Omega} u_{2n} \eta_2
\to
\int_{\Omega} u \eta_2
$ in $C[0,T]$.
Repeating these steps we can obtain $u_{mn}$ for any $m,n\in\mathbb{N}$.
Notice that for each $m$, the sequence $u_{nn}$ is a subsequence of $u_{mn}$ after finitely many terms. Thus, the convergence
\begin{equation}
\lim_{n\to\infty}
\sup_{t\in[0,T]} \abs{ \int_{\Omega} u_{nn} \eta_k - u \eta_k } =0
\end{equation}
holds for any  $k\in\mathbb{N}$.
The argument in \eqref{weak.topo.convergence.u} then follows the fact that 
$\{ \eta_1,  \eta_2 , \cdots \} $ is a basis of $ \mathbb{H}(\Omega)$.

Now we estimate $\partial_{t}B_{n}$. 
Picking $\varphi\in C_{c}^{\infty}(\Omega)$ with ${\rm div}\varphi =0$, we have
\begin{align*}
\int_{\Omega} \partial_{t}B_{n}\varphi
= I_1+ I_2 +I_3,
\end{align*}
where
\begin{align*} 
& I_1
=\int_{\Omega} u_n\nabla B_n \varphi, \\
& I_2
=\int_{\Omega} B_n\nabla u_n \varphi, \\
& I_3
=  \sigma \int_{\Omega} \nabla B_n :\nabla\varphi.
\end{align*}
The estimates are given as the following:
\begin{align*}
& |I_1|  \leq
\left\| u_n\right\| _{L^{3}(\Omega)}\left\| \nabla B_n\right\| _{L^2(\Omega)}\left\| \varphi\right\| _{L^{6}(\Omega)}
\leq 
C \left\Vert u_{n}\right\|_{L^2(\Omega)}^{\frac{1}{2}} 
\left\| u_{n}\right\|_{H_0^1(\Omega)}^{\frac{1}{2}} \left\| B_{n}\right\| _{H_0^1(\Omega)}\left\| \varphi\right\| _{H_0^1(\Omega)},  \\
& |I_2| \leq\left\| B_n \right\| _{L^{3}(\Omega)}\left\| \nabla u_n \right\| _{L^2(\Omega)}\left\| \varphi \right\| _{L^{6}(\Omega)}
\leq 
C \left\Vert B_{n}\right\| _{L^2(\Omega)}^{\frac{1}{2}}
\left\| B_{n}\right\| _{H_0^1(\Omega)}^{\frac{1}{2}}\left\| u_{n}\right\| _{H_0^1(\Omega)}\left\| \varphi\right\| _{H_0^1(\Omega)},  \\
& |I_3| \leq
C   \left\| \nabla B_n \right\| _{L^2(\Omega)}\left\| \nabla\varphi\right\| _{L^2(\Omega)}
\leq
C  \left\| B_n \right\| _{H_0^1(\Omega)}\left\| \varphi\right\| _{H_0^1(\Omega)}.
\end{align*}
Thus, we have
\begin{align*}
 \left\| \partial_{t}B_{n}\right\| _{H^{-1}(\Omega)} &  
 \leq 
C \left\Vert u_{n}\right\|_{L^2(\Omega)}^{\frac{1}{2}} 
\left\| u_{n}\right\|_{H_0^1(\Omega)}^{\frac{1}{2}} \left\| B_{n}\right\| _{H_0^1(\Omega)}
\\
 + &
C \left\Vert B_{n}\right\| _{L^2(\Omega)}^{\frac{1}{2}}
\left\| B_{n}\right\| _{H_0^1(\Omega)}^{\frac{1}{2}}\left\| u_{n}\right\| _{H_0^1(\Omega)}
+
C  \left\| B_n\right\| _{H_0^1(\Omega)}.
\end{align*}
Integrating with respect to $t$, we obtain
\begin{align*}
& \int_{0}^{T}  \left\| \partial_{t}B_{n}\right\| _{H^{-1}(\Omega)}  
 \leq
C 
\left\| u_{n}\right\|_{L^{\infty}([0,T];L^2(\Omega))}^{\frac{1}{2}} 
\int_{0}^{T}
\left\| u_{n}\right\|_{H_0^1(\Omega)}^{\frac{1}{2}} \left\| B_{n}\right\| _{H_0^1(\Omega)} 
{d}t \\
& +
C 
\left\| B_{n}\right\|_{L^{\infty}([0,T];L^2(\Omega))}^{\frac{1}{2}} 
\int_{0}^{T}
\left\| B_{n}\right\|_{H_0^1(\Omega)}^{\frac{1}{2}} \left\| u_{n}\right\| _{H_0^1(\Omega)} 
{d}t
+ 
C  \left\| B_n\right\| _{L^1([0,T];H_0^1(\Omega))}
\\
&\leq
C \left\| u_{n}\right\|_{L^{\infty}([0,T];L^2(\Omega))}^{\frac{1}{2}} 
\left\| u_{n}\right\|_{L^2([0,T];H_0^1(\Omega))}^{\frac{1}{2}} 
\left\| B_{n}\right\|_{L^2([0,T];H_0^1(\Omega))}
\\
& +
C \left\| B_{n}\right\|_{L^{\infty}([0,T];L^2(\Omega))}^{\frac{1}{2}} 
\left\| B_{n}\right\|_{L^2([0,T];H_0^1(\Omega))}^{\frac{1}{2}} 
\left\| u_{n}\right\|_{L^2([0,T];H_0^1(\Omega))}
+ C \left\| B_n\right\| _{L^2([0,T];H_0^1(\Omega))}.
\end{align*}
Thus, $\left\| \partial_t B_n\right\| _{L^1([0,T];H^{-1}(\Omega))}\leq C$.
By the Aubin-Lions lemma, we can find a suitable subsequence, still denoted by
$B_n$, such that
$B_{n}\to K$ strongly in $L^2([0,T];L^2(\Omega))$ 
for some $K\in L^2([0,T];L^2(\Omega))$.
We can prove the uniqueness of limits using the following argument.

\begin{Proposition}\label{uniqueness of limit}
For the strong and weak limits $B$, $G$ and $K$ of $B_n$, we have $B=G=K$.
\end{Proposition}

\begin{proof}
Since $ L^2([0,T];H_0^1(\Omega))\subseteq L^2([0,T];L^2(\Omega))$,
$B_n\rightharpoonup G$ in $L^2([0,T];H_0^1(\Omega))$ 
implies that $B_n\rightharpoonup G$ in $L^2([0,T];L^2(\Omega))$. 
By the uniqueness of the limit, we have $G=K$.
Since $ L^{\infty}([0,T];L^2(\Omega))\subseteq L^2([0,T];L^2(\Omega))$,
$B_n\rightharpoonup^{*} B$ in $L^{\infty}([0,T];L^2(\Omega))$
implies that  $B_n\rightharpoonup^{*} B$ in $L^2([0,T];L^2(\Omega))$.
Similarly to the first step, we have $B=K$. Thus, $B=G=K$. 
\end{proof}

We now consider the transport equation 
 $\partial_{t}\chi_n +u_n \cdot\nabla\chi_n=0$.
For $\varphi\in C_{c}^{\infty}(\Omega)$,
%the term $\int_{\Omega}\chi_{0}\psi(x,0)$ vanishes. 
we have 
\begin{align*}
\left|  \int_{\Omega} \partial_{t}\chi_{n} \varphi \right|
=\left| \int_{\Omega} \chi_{n}u_{n}\cdot\nabla\varphi \right| 
\leq
\left\| u_{n}\right\| _{L^2(\Omega)}\left\| \varphi\right\| _{H_0^1(\Omega)}.
\end{align*}
Thus, $\left\| \partial_{t}\chi_{n}\right\| _{H^{-1}(\Omega)}\leq\left\| u_{n}\right\| _{L^2(\Omega)}$,
which implies that 
$$
\left\| \partial_{t}\chi_{n}\right\| _{L^2([0,T];H^{-1}(\Omega))}
\leq
\left\| u_{n}\right\| _{L^2([0,T];H_0^1(\Omega))}
\leq C.
$$
Therefore, using the same argument as in \cite{Abels1}, Section 5.2, 
we can find a suitable subsequence, still denoted by $\chi_n$,
such that $\chi_{n}\to \chi$ strongly
in $L^2([0,T];L^2(\Omega))$.
Now we prove that $\zeta=\nabla\chi$ in the weak sense. 
For almost every $t\in[0,T]$
and
$\varphi\in C_{c}^{\infty}(\Omega)$, we have
\[
\int_{\Omega}\chi_{n}{\rm div}\varphi {d}x=-\int_{\Omega}\varphi\cdot{d}\nabla\chi_{n}
\]
for any $n$.
Since ${\rm div}\varphi\in L^{1}(\Omega)$,
we have
\begin{align*}
\int_{\Omega}\chi_{n}(t){\rm div}\varphi {d}x
& \longrightarrow\int_{\Omega}\chi(t){\rm div}\varphi{d}x,  \\
\int_{\Omega}\varphi \cdot {d}\nabla\chi_{n}(t) & \longrightarrow\int_{\Omega} \varphi \cdot{d}\zeta(t),
\end{align*}
which implies that $\nabla\chi=\zeta$ for almost every $t\in [0,T]$.

Therefore, we finally obtain
\begin{align*}
u_{n}\rightharpoonup^{*}u & \ \ \  \text{in}\,\,L^{\infty}([0,T];L^2(\Omega)),\\
u_{n}\rightharpoonup u & \ \ \  \text{in}\,\,L^2([0,T];H_0^1(\Omega)),\\
u_n \to u& \ \ \  \text{in}\,\, C([0,T]; \mathbb{V}^*(\Omega)),\\
B_{n}\rightharpoonup^{*}B &  \ \ \  \text{in} \ L^{\infty}([0,T];L^2(\Omega)),\\
B_{n}\rightharpoonup B & \ \ \  \text{in} \ L^2([0,T];H_0^1(\Omega)),\\
B_{n}\to B & \ \ \  \text{in} \ L^2([0,T];L^2(\Omega)),\\
\chi_{n}\rightharpoonup^{*}\chi & \ \ \  \text{in} \ L^{\infty}([0,T];L^{\infty}(\Omega)),\\
\nabla\chi_{n}\rightharpoonup^{*}\nabla\chi & \ \ \  \text{in} \ 
L^{\infty}([0,T];H^{-3}(\Omega)),\\
\chi_{n} \to \chi & \ \ \  \text{in} \ L^2 ([0,T];L^2(\Omega)).\\
\end{align*}

\subsection{Varifold limit of $H_{\chi_{n}}$}  \label{vari limit}
Since we cannot pass the limit directly when dealing with $H_{\chi_n}$,
we will represent them with 
varifolds, and then consider the weak limit of the varifolds. 
Recall that $\Gamma_{k}(t):=X_{u_k}(t,\Gamma_0)$,
where $\Gamma_0=\partial\Omega_0^+$.
Using the same argument as in \cite{Abels1},
we consider the varifold $V_{k}(t)$ corresponding to $\Gamma_{k}(t)$,
i.e.
\begin{align*}
\left\langle V_k (t),\varphi \right\rangle  
:= \int_{\Omega}\varphi\left(  x, n_k (t,x)\right){d}\left|\nabla\chi_k (t)\right|
\end{align*}
for any
$\varphi\in C_0(\Omega \times\mathbb{S}^{2})$,
where $n_k (t,x):= \nabla\chi_k (t,x) / \left|\nabla\chi_k (t,x)\right| $.
Since
$$
\left|\left\langle V_k (t),\varphi \right\rangle \right|
\leq
\left\| \nabla\chi_k (t) \right\|_{\mathcal{M}(\Omega)}
\left\| \varphi \right\|_{C_0(\Omega\times\mathbb{S}^{2})},
$$
we have
$$ \left\| V_k (t) \right\|_{\mathcal{M}(\Omega\times\mathbb{S}^2)} 
\leq \left\| \nabla\chi_k (t) \right\|_{\mathcal{M}(\Omega)}.$$
Now for all $\varphi\in L^1([0,T];C_0(\Omega\times\mathbb{S}^{2}))$, we define
\begin{align*}
\left\langle V_{k},\varphi\right\rangle  & :=\int_{0}^{T}\int_{\Omega}\varphi\left(t, x,n_k (t,x)\right){d}\left|\nabla\chi_{k}(t)\right| {d}t.
\end{align*}
%For $V_{k}$, we have
%\begin{align*}
%\left|\left\langle V_{k},\varphi\right\rangle \right|  
%\leq
%\int_{0}^{T} 
%\left\| \nabla\chi_k (t) \right\|_{\mathcal{M}(\Omega)}
%\left\| \varphi(t) \right\|_{C_0(\Omega\times\mathbb{S}^{2})}{d}t
%\\
%\leq
%\left\| \nabla\chi_{k}\right\| _{L^{\infty}([0,T];\mathcal{M}(\Omega))}\left\| \varphi\right\| _{L^{1}([0,T];C_0(\Omega\times\mathbb{S}^{2}))},
%\end{align*}
%and
Then we have
$$
\left\| V_{k}\right\| _{L_w^{\infty}([0,T];\mathcal{M}(\Omega\times\mathbb{S}^2))} \leq\left\| \nabla\chi_{k}\right\|_{L_w^{\infty}([0,T];\mathcal{M}(\Omega))}.
$$
Since $V_{k}$ is bounded in $L_w^{\infty}([0,T];\mathcal{M}(\Omega\times\mathbb{S}^{2})$
and $\mathcal{M}(\Omega\times\mathbb{S}^2)
\hookrightarrow H^{-3}(\Omega\times\mathbb{S}^2)$,
by the same argument as in \cite{Abels1},
there exists $V\in L_w^{\infty}([0,T];\mathcal{M}(\Omega\times\mathbb{S}^{2}))$,
such that
\begin{align*}
V_{k}\rightharpoonup^{*}V & \ \ \  \text{in} \ L^{\infty}([0,T];H^{-3}(\Omega\times\mathbb{S}^{2})).
\end{align*}
For $\psi\in C_{c}^{\infty}([0,T)\times\Omega)$, letting $\varphi(x,s,t)={\rm div}\psi-s\otimes s:\nabla\psi$,
we have $\varphi\in L^{1}([0,T];C_0(\Omega\times\mathbb{S}^{2}))$.
The regularity of $n_{k}(x):=-\nabla\chi_{k}/\left|\nabla\chi_{k}\right|$
is guaranteed since $\nabla\chi_{k}$ are approximating solutions.
Thus, we have 
\begin{equation}
\begin{aligned}
  -&\int_{0}^{T}\left\langle H_{\chi_{k}(t)},\psi(t)\right\rangle =\int_{0}^{T}\int_{\Omega}P_{\tau}:\nabla\psi{d}\left|\nabla\chi_{k}\right|{d}t\\
= & \int_{0}^{T}\int_{\Omega}\bigl({\rm div}\psi-n_{k}\otimes n_{k}:\nabla\psi\bigr){d}\left|\nabla\chi_{k}\right|{d}t\\
= & \int_{0}^{T}\int_{\Omega\times \mathbb{S}^{2}}\bigl({\rm div}\psi-s\otimes s:\nabla\psi\bigr){d}V_k{d}t\\
\to & \int_{0}^{T}\int_{\Omega\times \mathbb{S}^{2}}\bigl({\rm div}\psi-s\otimes s:\nabla\psi\bigr){d}V{d}t\\ 
%= & \int_{0}^{T}\int_{\Omega}\int_{\mathbb{S}^{2}}\bigl({\rm div}\psi-s\otimes s:\nabla\psi\bigr){d}\delta_{n_{k}(x,t)}{d}\left|\nabla\chi_{k}\right|{d}t\\
= & \int_{0}^{T}\left\langle \delta V(t),\psi(t)\right\rangle {d}t\\
%\to & \int_{0}^{T}\left\langle \delta V(t),\psi(t)\right\rangle {d}t,
\end{aligned}
\end{equation}
as $k\to \infty$.
%The reason that $V_k$ can be decomposed into $\delta_{n_k}$ and $|\nabla\chi_k|$
%can be found in \cite{Abels1}, Section 3.
Letting $\varphi(x,s,t)=s\psi(x,t)$, 
we have
\begin{align*}
&\int_{0}^{T}\left\langle \nabla\chi_{k},\psi\right\rangle   =-\int_{0}^{T}\int_{\partial^{*}\{\chi_{k}=1\}}\psi\cdot n_{k}{d}\mathcal{H}^{2}{d}t\\
= & -\int_{0}^{T}\int_{\Omega}\psi\cdot n_{k}{d}\left|\nabla\chi_{k}\right|{d}t
=  -\int_{0}^{T}\int_{\Omega\times\mathbb{S}^{2}}\psi\cdot s{d}V_{k}{d}t.
\end{align*}
Since $\nabla\chi_{k}\rightharpoonup^{*}\nabla\chi$ in $L^{\infty}([0,T];H^{-3}(\Omega))$
and $V_{k}\rightharpoonup^{*}V$ in $L^{\infty}([0,T];H^{-3}(\Omega\times\mathbb{S}^{2}))$,
we have
\begin{align*}
\int_{0}^{T}\left\langle \nabla\chi_{k},\psi\right\rangle  & \to\int_{0}^{T}\left\langle \nabla\chi,\psi\right\rangle, \\
\int_{0}^{T}\int_{\Omega\times\mathbb{S}^{2}}\psi\cdot s{d}V_{k}{d}t & \to\int_{0}^{T}\int_{\Omega\times\mathbb{S}^{2}}\psi\cdot s{d}V{d}t.
\end{align*}
Therefore, using the fact that $C_c^{\infty}(\Omega)$ is dense in $C_0(\Omega)$,
the equation
\begin{equation}\label{vari solution def 1}
\begin{aligned}
\int_{\Omega\times\mathbb{S}^{2}}\psi\cdot s{d}V 
& = - \int_{\Omega}\psi{d}\nabla\chi.
\end{aligned}
\end{equation}
holds for all $\psi\in C_0(\Omega)$.

\subsection{Passing the limit}

The weak formula represented by the varifolds is
\begin{align*}
(\partial_{t}u_{n},\varphi)_{Q_{T}} & -(u_{n}\otimes u_{n},\nabla\varphi)_{Q_{T}}+(B_{n}\otimes B_{n},\nabla\varphi)_{Q_{T}}\\
+ & 2(\nu(\chi_{n})Du_{n},D\varphi)_{Q_{T}}
+
\kappa\int_{0}^{T}\left\langle \delta V_n (t),\varphi(t)\right\rangle {d}t
=0,
\end{align*}
where $\varphi\in C_c^{\infty}([0,T)\times\Omega)$ and ${\rm div}\varphi=0$.
%Since $u_{n}\to u$ strongly in $L^2([0,T];L^2(\Omega))$,
%we have
%\begin{align*}
% & \left|(u_{n}\otimes u_{n},\nabla\varphi)_{Q_{T}}-(u\otimes u,\nabla\varphi)_{Q_{T}}\right|\\
%& \leq  \left|(u_{n}\otimes u_{n},\nabla\varphi)_{Q_{T}}-(u_{n}\otimes u,\nabla\varphi)_{Q_{T}}\right|+\left|(u_{n}\otimes u,\nabla\varphi)_{Q_{T}}-(u\otimes u,\nabla\varphi)_{Q_{T}}\right|\\
%& \leq  \left\| u_{n}\right\| _{L^2(Q_T)}\left\| u_{n}-u\right\| _{L^2(Q_T)}\left\| \nabla\varphi\right\| _{L^{\infty}(Q_T)}
%+\left\| u\right\| _{L^2(Q_T)}\left\| u_{n}-u\right\| _{L^2(Q_T)}\left\| \nabla\varphi\right\| _{L^{\infty}(Q_T)}.
%\end{align*}
We have
\begin{equation}
\begin{aligned}
&
\abs{\int_0^T \int_{\Omega} ( u_n\otimes u_n - u\otimes u) : \nabla \varphi}
\\
&
\leq 
\abs{\int_0^T \int_{\Omega}  u_n\otimes( u_n - u) : \nabla \varphi}
+
\abs{\int_0^T \int_{\Omega} ( u_n - u) \otimes u : \nabla \varphi}
=: I_1 + I_2.
\end{aligned}
\end{equation}
In  $I_1$, the integrand  equals to
$( u_n - u) \cdot  \left( u_n \nabla \varphi \right)$. 
Since $u_n$ and $\varphi$ are smooth, by direct calculation 
we have ${\rm div}(u_n \nabla \varphi)=0$,
which implies $u_n \nabla \varphi \in  \mathbb{V}(\Omega)$.
Recall that $\norm{\cdot}_{\mathbb{V}}:= \norm{\cdot}_{H_0^1}$ and
$u_n$ is bounded in $L^2([0,T];H_0^1(\Omega))$, 
so we have
\begin{equation}
\begin{aligned}
&
I_1
=
\int_0^T \langle u_n-u, \psi \rangle_{\mathbb{V}^*,\mathbb{V}}
\leq   
\norm{u_n-u}_{L^{\infty}\mathbb{V}^*}\norm{u_n \nabla \varphi}_{L^1 \mathbb{V}}
\\
&\leq
C\norm{u_n-u}_{L^{\infty}\mathbb{V}^*}\norm{u_n \nabla \varphi}_{L^2 H_0^1} \to 0.
\end{aligned}
\end{equation}
%and
%\begin{equation}
%\norm {u_n \nabla \varphi}_{L^1([0,T];V(\Omega))} \leq \norm {u_n \nabla \varphi}_{L^2([0,T];H_0^1(\Omega))}\leq C.
%\end{equation}
The second term goes to 0 since $u_n \rightharpoonup u$ weakly in $L^2([0,T];H_0^1)$
 and $u\nabla\varphi\in L^2([0,T];H_0^1)$.
 
For the term $(B_{n}\otimes B_{n},\nabla\varphi)_{Q_{T}}$,
we have
\begin{align*}
 & \left| \int_0^T\int_{\Omega} B_{n}\otimes B_{n}: \nabla\varphi
 - B\otimes B: \nabla\varphi  dxdt   \right|\\
\leq & \left\| B_{n}\right\| _{L^2(Q_T)}\left\| B_{n}-B\right\| _{L^2(Q_T)}\left\| \nabla\varphi\right\| _{L^{\infty}(Q_T)}+\left\| B\right\| _{L^2(Q_T)}\left\| B_{n}-B\right\| _{L^2(Q_T)}\left\| \nabla\varphi\right\| _{L^{\infty}(Q_T)},
\end{align*}
which converges to 0.
For the viscosity term, we have
\begin{align*}
 & \left|(\nu(\chi_{n})Du_{n},D\varphi)_{Q_{T}}-(\nu(\chi)Du,D\varphi)_{Q_{T}}\right|\\
= & \left|\int_{0}^{T}\int_{\Omega}\left(\chi_{n}\nu^{+}Du_{n}+(1-\chi_{n})\nu^{-}Du_{n}\right):D\varphi-\int_{0}^{T}\int_{\Omega}\left(\chi\nu^{+}Du+(1-\chi)\nu^{-}Du\right):D\varphi\right|\\
\leq & \nu^{+}\left|\int_{0}^{T}\int_{\Omega}\chi_{n}Du_{n}:D\varphi-\chi Du:D\varphi\right|+\nu^{-} \left| \int_{0}^{T}\int_{\Omega}(1-\chi_{n})Du_{n}:D\varphi-(1-\chi)Du:D\varphi \right|.
\end{align*}
We only need to show the convergence of the first term,
since the second term can be shown by the same argument.
For the first term, we have
\begin{equation}\label{Du weak converge in L2}
\begin{aligned}
 & \left|\int_{0}^{T}\int_{\Omega}\chi_{n}Du_{n}:D\varphi-\chi Du:D\varphi \right|\\
\leq & \int_{0}^{T}\int_{\Omega}\left|\chi_{n}-\chi \right| \left| Du_{n} \right| \left| D\varphi \right| +\left|\int_{0}^{T}\int_{\Omega}Du_{n}:(\chi D\varphi)-\int_{0}^{T}\int_{\Omega}Du:(\chi D\varphi)\right|.
\end{aligned}
\end{equation}
Since $u_{n}\rightharpoonup u$ in $L^2([0,T];H_0^1(\Omega))$
and $\varphi\in C_{c}^{\infty}([0,T)\times\Omega)$, we have $Du_{n}\rightharpoonup Du$
in $L^2([0,T];L^2(\Omega))$ and $\chi D\varphi\in L^2([0,T];L^2(\Omega))$.
Thus, the second term goes to 0.
Since $\bigl|Du_{n}\bigr|\bigl|D\varphi\bigr|\in L^2([0,T];L^2(\Omega))$
and $\chi_{n}\to\chi$ strongly in $L^2([0,T];L^2(\Omega))$, we
have 
\[
\int_{0}^{T}\int_{\Omega}\bigl|\chi_{n}-\chi\bigr|\bigl|Du_{n}\bigr|\bigl|D\varphi\bigr|
\leq
C\left\| \chi_{n}-\chi\right\| _{L^2L^2}\to 0,
\]
which finishes the proof.

For the transport equation,
\begin{align*}
(\chi_{0},\varphi(x,0))_{\Omega}+(\chi_{n},\partial_{t}\varphi)_{Q_{T}}+(\chi_{n},u_{n}\cdot\nabla\varphi)_{Q_{T}}= & 0.
\end{align*}
Since $\partial_t \varphi\in L^{1}([0,T];L^{1}(\Omega))$, we have $(\chi_{n},\partial_{t}\varphi)_{Q_{T}}\to(\chi,\partial_{t}\varphi)_{Q_{T}}$.
For the third term, we have
\begin{align*}
 & \left| \int_{0}^{T}\int_{\Omega}\chi_{n}u_{n}\cdot\nabla\varphi-\int_{0}^{T}\int_{\Omega}\chi u\cdot\nabla\varphi \right| \\
\leq & \left|\int_{0}^{T}\int_{\Omega}\chi_{n}u_{n}\cdot\nabla\varphi-\int_{0}^{T}\int_{\Omega}\chi_{n}u\cdot\nabla\varphi\right|+\left|\int_{0}^{T}\int_{\Omega}\chi_{n}u\cdot\nabla\varphi-\int_{0}^{T}\int_{\Omega}\chi u\cdot\nabla\varphi\right|\\
\leq & \left\| \chi_{n}\right\| _{L^2(Q_T)}\left\| u_{n}-u\right\| _{L^2(Q_T)}\left\| \nabla\varphi\right\| _{L^{\infty}(Q_T)}+\left\| \chi_{n}-\chi\right\| _{L^2(Q_T)}\left\| u\right\| _{L^2(Q_T)}\left\| \nabla\varphi\right\| _{L^{\infty}(Q_T)},
\end{align*}
which goes to $0$ as $n\to \infty$.

From Section \ref{vari limit}, we have proved that
$$
\int_{0}^{T}\left\langle \delta V_n (t),\varphi\right\rangle {d}t
\to  \int_{0}^{T}\left\langle \delta V(t),\varphi\right\rangle {d}t.
$$
Therefore, by letting $n\to \infty$, we obtain
\begin{align*}
-(u_{0},\varphi(0))_{\Omega}-(u,\partial_{t}\varphi)_{Q_{T}} & -(u\otimes u,\nabla\varphi)_{Q_{T}}+(B\otimes B,\nabla\varphi)_{Q_{T}}\\
+ & 2(\nu(\chi)Du,D\varphi)_{Q_{T}}+\kappa\int_{0}^{T}\left\langle \delta V(t),\varphi(t)\right\rangle {d}t=0
\end{align*}
 for all $\varphi\in C_{c}^{\infty}([0,T)\times\Omega)$ with ${\rm div}\varphi=0$.
From \eqref{vari solution def 1}, we have
\begin{align*}
\int_{\Omega\times\mathbb{S}^{2}}\psi\cdot s{d}V & =-\int_{\Omega}\psi{d}\nabla\chi
\end{align*}
for all $\psi\in C_0(\Omega)$.
From Proposition 2.2 in \cite{Abels1},
$\chi$ is the unique renormalized solution of 
\begin{align*}
&\partial_{t}\chi+u\cdot\nabla\chi=0 \ \ \ \text{in }Q_{T}, \\
& \chi|_{t=0}=\chi_{0}\ \ \  \text{in }\Omega.
\end{align*}
It remains to prove that
$(u,B,\chi,V)$ satisfies the generalized energy inequality.
Since $ u_n \to u$, $ B_n\to B$ in $L^2([0,T];L^2(\Omega))$.
For  suitable subsequences,
we have $ u_n (t)\to u(t)$, $ B_n (t)\to B(t)$ in $L^2(\Omega)$ 
for almost every $t\in[0,T]$.
From Theorem 1.1.1 in \cite{Evans2}, we have
$$
\left\| u (t)\right\|_{L^2}\leq 
\liminf_{n\to\infty} \left\| u_n (t)\right\|_{L^2},
$$
$$
\left\| B (t)\right\|_{L^2}\leq 
\liminf_{n\to\infty} \left\| B_n (t)\right\|_{L^2}.
$$
In Section \ref{vari limit} we have proved 
$ \left\| V_n (t) \right\|_{\mathcal{M}(\Omega\times\mathbb{S}^2)} 
\leq \left\| \nabla\chi_n (t) \right\|_{\mathcal{M}(\Omega)} $.
Given any $\varphi\in C_0(\Omega)$, we have
\begin{align*} 
\left| \int_{\Omega\times\mathbb{S}^2}\varphi {d}V (t) \right| 
=
&
\left| \lim_{n\to\infty} \int_{\Omega\times\mathbb{S}^2}\varphi {d}V_n (t)  \right| 
\leq 
\liminf_{n\to\infty} \left\| \varphi \right\|_{L^{\infty}}  
\left\| V_n (t) \right\|_{\mathcal{M}(\Omega\times\mathbb{S}^2)}   
\\
&
\leq 
\liminf_{n\to\infty} \left\| \varphi \right\|_{L^{\infty}}  
\left\| \nabla\chi_n (t) \right\|_{\mathcal{M}(\Omega)}.
\end{align*}
Thus, 
$$\left\| V(t) \right\|_{{\mathcal{M}(\Omega\times\mathbb{S}^2)}} 
\leq 
\liminf_{n\to\infty} \left\| \nabla\chi_n (t) \right\|_{\mathcal{M}(\Omega)}.$$
Since $\nabla B_n \rightharpoonup \nabla B$ in $L^2([0,T];L^2(\Omega))$.
For all $t\in[0,T]$, we still have $\nabla B_n \rightharpoonup \nabla B$ in $L^2([0,t];L^2(\Omega))$.
Thus,
$$ 
\int_0^t \left\| \nabla B \right\|_{L^2}^2 
\leq 
\liminf_{n\to\infty} \int_0^t \left\| \nabla B_n \right\|_{L^2}^2.
$$
For the viscosity term, notice that
\begin{align*}
&f_n:=\left(\nu(\chi_n)Du_n - \nu(\chi_n)Du\right):(Du_n-Du)
\\
&=\left( \chi_n \nu^+  + (1-\chi_n) \nu^- \right) \left| Du_n-Du \right|^2
\geq 0.
\end{align*}
From $Du_n \rightharpoonup Du $ in $L^2([0,T];L^2(\Omega))$, we have
$$\int_0^t \int_{\Omega} f_n \leq C.$$ 
Thus, using Fatou's lemma, we obtain
\begin{equation}\label{lower cont viscosity term}
\begin{aligned}
&0\leq \int_0^t \int_{\Omega} \liminf_{n\to\infty} f_n 
\leq  \liminf_{n\to\infty} \int_0^t \int_{\Omega} f_n
\\
&=
\liminf_{n\to\infty} \int_0^t \int_{\Omega} \nu(\chi_n) Du_n:Du_n
+\lim_{n\to\infty} \int_0^t \int_{\Omega} -\nu(\chi_n) Du_n:Du
\\
&
+\lim_{n\to\infty} \int_0^t \int_{\Omega} -\nu(\chi_n) Du:Du_n
+\lim_{n\to\infty} \int_0^t \int_{\Omega} \nu(\chi_n) Du:Du
\\
&=
\liminf_{n\to\infty} \int_0^t \int_{\Omega} \nu(\chi_n) Du_n : Du_n
-\int_0^t \int_{\Omega} \nu(\chi)Du:Du.
\end{aligned}
\end{equation}
This is because \eqref{Du weak converge in L2} yields
$ \chi_n Du_n \rightharpoonup \chi Du $ in $L^2([0,T];L^2(\Omega))$,
which implies that
\begin{align*}
&\lim_{n\to\infty} \int_0^t \int_{\Omega} -\nu(\chi_n)Du_n:Du
=
\lim_{n\to\infty} \int_0^t \int_{\Omega} -\nu(\chi_n)Du:Du_n
\\
&=
-\int_0^t \int_{\Omega} \nu(\chi)Du:Du.
\end{align*}
Since $|Du|^2\in L^1([0,T];L^1(\Omega))$ 
and $\chi_n \rightharpoonup^* \chi$ in  $L^{\infty}([0,T];L^{\infty}(\Omega))$,
we have
\begin{align*}
\lim_{n\to\infty} \int_0^t \int_{\Omega} \nu(\chi_n)Du:Du
=
\int_0^t \int_{\Omega} \nu(\chi)Du:Du.
\end{align*}
Thus, the lower-semicontinuity in \eqref{lower cont viscosity term} has been proved.

Recall that \eqref{Xn energy u} and \eqref{Xn energy B} give us
\begin{equation}\label{final energy ineq}
\begin{aligned}
\frac{1}{2}  \left\| u_n (t)\right\| _{L^2}^{2} 
& +\frac{1}{2}\left\| B_n (t)\right\| _{L^2}^{2}+\kappa\left\| \nabla\chi_n (t)\right\| _{\mathcal{M}}
+ 2 \int_0^t \int_{\Omega} \nu(\chi_n)Du_n:Du_n   {d}x{d}s
\\
&
+\sigma\int_0^t \left\| \nabla B_n \right\| _{L^2}^{2} {d}s
\leq\frac{1}{2}\left\| u_{0}\right\| _{L^2}^{2}+\frac{1}{2}\left\| B_{0}\right\| _{L^2}^{2}+\kappa\left\| \nabla\chi_{0}\right\| _{\mathcal{M}}.
\end{aligned}
\end{equation}
Taking the $\liminf$ on \eqref{final energy ineq},  and using the fact that $\liminf a_n + \liminf b_n \leq \liminf (a_n+b_n)$,
we finally obtain
\begin{equation}
\begin{aligned}
\frac{1}{2}  \left\| u(t)\right\| _{L^2}^{2} 
& +\frac{1}{2}\left\| B(t)\right\| _{L^2}^{2}+\kappa\left\| V(t)\right\| _{\mathcal{M}(\Omega\times\mathbb{S}^2)}
+ 2 \int_0^t \int_{\Omega} \nu(\chi)Du:Du  {d}x{d}s
\\
+
&
\sigma\int_0^t \left\| \nabla B\right\| _{L^2}^{2} {d}s
\leq\frac{1}{2}\left\| u_{0}\right\| _{L^2}^{2}+\frac{1}{2}\left\| B_{0}\right\| _{L^2}^{2}+\kappa\left\| \nabla\chi_{0}\right\| _{\mathcal{M}},
\end{aligned}
\end{equation}
which finishes the proof of the energy inequality.

Therefore, we have finished the proof of Theorem \ref{main theorem}.
\bigskip

\bigskip\bigskip

\section*{Acknowledgments}
I would like to thank Dr. Dehua Wang, Dr. Armin Schikorra, Dr. Ming Chen and Dr. Ian Tice
for valuable comments and suggestions as well as helpful discussions.
 I would also like to thank Dr. Helmut Abels for 
careful explanations of the techniques in \cite{Abels1}.
I really appreciate that!

\end{document}